\documentclass{cmslatex}
\usepackage{latexsym, amssymb, enumerate, amsmath}

\usepackage[usenames,dvipsnames]{color}

\usepackage{graphicx,epsfig}
\usepackage{float}
\usepackage{subfigure}

\usepackage{multirow}

\usepackage[margin=30pt,font=small,labelsep=endash]{caption}

\usepackage[english]{babel}

\usepackage[
          hypertex,
          ,pdftex
          ,pdfproducer={Latex with hyperref}
          ,pdfcreator={pdflatex}
          ]{hyperref}
          
\renewcommand{\theequation}{\arabic{section}.\arabic{equation}}

\newcommand{\fabrice}[1]{\textcolor{black}{#1}}

\def\MM{\mathbb{M}}

\def\RR{\mathbb{R}}

\def\be{\begin{equation}}
\def\ee{\end{equation}}

\let\ds\displaystyle
\let\eps\varepsilon

\newtheorem{thm}{Theorem}[section]
\newtheorem{prop}[thm]{Proposition}
\newtheorem{lem}[thm]{Lemma}

\newtheorem{rem}[thm]{Remark}

\begin{document}

\title{Duality-based Asymptotic-Preserving method for highly
  anisotropic diffusion equations}

\author{Pierre Degond\footnotemark[2]\ \footnotemark[3] \and  Fabrice
  Deluzet\footnotemark[2]\ \footnotemark[3] 
  \and Alexei Lozinski\footnotemark[2]
  \and Jacek Narski\footnotemark[2] \and Claudia Negulescu\footnotemark[4] } 

\renewcommand{\thefootnote}{\fnsymbol{footnote}}

\footnotetext[2]{Universit\'e de Toulouse, UPS, INSA, UT1, UTM, Institut de Math\'ematiques de Toulouse, F-31062 Toulouse, France}
\footnotetext[3]{CNRS, Institut de Math\'ematiques de Toulouse UMR 5219, F-31062 Toulouse, France}
\footnotetext[4]{CMI/LATP, Universit\'e de Provence, 39 rue Fr\'ed\'eric Joliot-Curie 13453 Marseille cedex 13} %({\tt claudia.negulescu@cmi.univ-mrs.fr})}

\renewcommand{\thefootnote}{\arabic{footnote}}
\maketitle

%\tableofcontents
\begin{abstract}
  The present paper introduces an efficient and
  accurate numerical scheme for the solution of a highly anisotropic
  elliptic equation, the anisotropy direction being given by a
  variable vector field. This scheme is based on an asymptotic
  preserving reformulation of the original system, permitting an
  accurate resolution independently of the anisotropy strength and
  without the need of a mesh adapted to this anisotropy. The
  counterpart of this original procedure is the larger system size, enlarged by adding auxiliary variables
  and Lagrange multipliers. This Asymptotic-Preserving method generalizes the method investigated in
  a previous paper \cite{DDN} to the case of an arbitrary anisotropy
  direction field.
\end{abstract}
% Modify for CMS
%{\abstract{}}
% END OF Modify for CMS
%%%%%%%%%%%%%%%%%%%%%%%
\section{Introduction}
%%%%%%%%%%%%%%%%%%%%%%%

Anisotropic problems are common in mathematical modeling of physical
problems. They occur in various fields of applications such as flows
in porous media \cite{porous1,TomHou}, semiconductor modeling
\cite{semicond}, quasi-neutral plasma simulations \cite{Navoret},
image processing \cite{image1,Weickert}, atmospheric or oceanic flows
\cite{ocean} and so on, the list being not exhaustive. The initial
motivation for this work is closely related to magnetized plasma
simulations such as atmospheric plasma \cite{Kelley2,Kes_Oss},
internal fusion plasma \cite{Beer,Sangam} or plasma thrusters
\cite{SPT}. In this context, the media is structured by the magnetic
field, which may be strong in some regions and weak in others. Indeed,
the gyration of the charged particles around magnetic field lines
dominates the motion in the plane perpendicular to magnetic
field. This explains the large number of collisions in the
perpendicular plane while the motion along the field lines is rather
undisturbed. As a consequence the mobility of particles in different
directions differs by many orders of magnitude. This ratio can be as
huge as $10^{10}$. On the other hand, when the magnetic field is weak
the anisotropy is much smaller. As the regions with weak and strong
magnetic field can coexist in the same computational domain, one needs
a numerical scheme which gives accurate results for a large range of
anisotropy strengths. The relevant boundary conditions in many fields
of application are periodic (for instance in simulations of the
tokamak plasmas on a torus) or Neumann boundary conditions
(atmospheric plasma for example \cite{BCDDGT_1}). For these reasons
we propose a strongly anisotropic model problem for wich we wish to
introduce an efficient and accurate numerical scheme. This model
problem reads
\begin{gather}
  \left\{ 
    \begin{array}{ll}
      - \nabla \cdot \mathbb A \nabla \phi^{\varepsilon } = f       & \text{ in }
      \Omega, \\[3mm]
%      {1 \over \varepsilon} 
%      n_\parallel \cdot 
%      \left( A_\parallel  b \otimes b \nabla \phi^{\varepsilon } \right)
%      +
%      n_\perp \cdot 
%      \left(A_\perp (Id - b \otimes b)\nabla \phi^{\varepsilon }\right) 
%      = 0
%      & \text{ on } \partial\Omega _N,  \\[3mm]
n \cdot \mathbb A \nabla \phi^{\varepsilon }= 0
      & \text{ on } \partial\Omega _N\,,\\[3mm]
      \phi^{\varepsilon }= 0
      & \text{ on } \partial\Omega _D\,,
    \end{array}
  \right.
  \label{eq:Jg0a}
\end{gather}
where $\Omega \subset \RR^{2}$ or $\Omega \subset \RR^{3}$ is a
bounded domain with boundary $\partial \Omega = \partial\Omega _D
\cup \partial\Omega _N$ and outward normal $n$. The direction of the
anisotropy is defined by a vector field $B$, where we suppose
$\text{div} B = 0$ and $B \neq 0$. The direction of $B$ is given by a
vector field $b = B/|B|$. The anisotropy matrix is then defined as
\begin{gather} \mathbb A = \frac{1}{\varepsilon }
  A_\parallel b \otimes b + (Id - b \otimes b)A_\perp (Id - b \otimes
  b)
  \label{eq:Jh0a}
\end{gather}
and $\partial \Omega _D = \{ x \in \partial\Omega \ | \ b (x) \cdot n
= 0 \}$. The scalar field $A_\parallel>0$ and the symmetric positive
definite matrix field $A_\perp$ are of order one while the parameter
$0 < \varepsilon < 1$ can be very small, provoking thus the high
anisotropy of the problem. This work extends the results of
\cite{DDN}, where the special case of a vector field $b$, aligned with
the $z$-axis, was studied. 
\color{black}
An extension of this approach is proposed
in \cite{besse} to handle more realistic anisotropy topologies.  It
relies on the introduction of a curvilinear coordinate system with one
coordinate aligned with the anisotropy direction. Adapted
 coordinates are widely used in the framework of plasma
simulation (see for instance \cite{Beer, Dhaeseleer, Miamoto}), 
coordinate systems being either developped to fit particular magnetic field
geometry or plasma equilibrium (Euler potentials \cite{Stern}, toroidal and poloidal
 \cite{Gysela, poltor}, quasiballooning \cite{Dimits}, Hamada \cite{Hamada} and Boozer
\cite{Boozer} coordinates). Note that the study of certain plasma regions in
a tokamak have motivated the use of non-orthogonal coordinates systems \cite{Igitkhanov}. 
In contrast with all these methods, we propose here a numerical scheme that uses coordinates and meshes independent of the
anisotropy direction, like in \cite{Ottaviani}. This
feature offers the capability to easily treat time evolving
anisotropy directions. This \fabrice{is very} important in the context
of tokamak plasma simulation, the anisotropy being driven by the
magnetic field which is time dependent.
\color{black}

One of the difficulties associated with the numerical solution of
problem (\ref{eq:Jg0a}) lies in the fact that this problem becomes
very ill-conditioned for small $0 < \varepsilon \ll 1$. Indeed,
replacing $\varepsilon $ by zero yields an ill-posed problem as it has
an infinite number of solutions (any function constant along the $b$
field solves the problem with $\varepsilon =0 $). In the discrete case
the problem translates into a linear system which is ill-conditioned,
as it mixes the terms of different orders of magnitude for
$\varepsilon \ll 1 $. As a consequence the numerical algorithm for
solving this linear system gives unacceptable errors (in the case of
direct solvers) or fails to converge in a reasonable time (in the case
of iterative methods).

This difficulty arises when the boundary conditions supplied to the
dominant $O (1/\varepsilon )$ operator lead to an ill-posed
problem. This is the case for Neumann boundary
conditions imposed on the part of the boundary with $b \cdot n \neq 0$ as well as for periodic boundary conditions. 
If instead, the boundary conditions are such that the dominant operator
gives a well-posed problem, the numerical difficulty vanishes.  One
can resort to standard methods, as the dominant operator is sufficient
to determine the limit solution. This is the case for Dirichlet and
Robin boundary conditions. The problem addressed in this paper arises
therefore only with specific boundary conditions. It has however a
considerable impact in numerous physical problems concerning plasmas,
geophysical flows, plates and shells as an example. In this paper, we
will focus on Neumann boundary condition since they represent a larger
range of physical applications. The periodic boundary
conditions can be addressed in a very similar way.

Numerical methods for anisotropic problems have been extensively
studied in the literature. Distinct methods have been
developed. Domain decomposition techniques using multiple coarse grid
corrections are adapted to the anisotropic equations in
\cite{Giraud,Koronskij}. Multigrid methods have been studied in
\cite{Gee,Notay}. For anisotropy aligned with one or two directions,
point or plane smoothers are shown to be very efficient
\cite{ICASE}. The $hp$-finite element method is also known to give
good results for singular perturbation problems \cite{Melenek}. All
these methods have in common that they try to discretize the
anisotropic PDE as it is written and then to apply purely numerical
tricks to circumvent the problems related to lack of accuracy of the
discrete solution or to the slow convergence of iterative
algorithms. This leads to methods which are rather difficult to
implement.

The approach that we pursue in this paper is entirely different: we
reformulate first the original PDE in such a way that the resulting
problem can be efficiently and accurately discretized by
straight-forward and easily implementable numerical methods for any
anisotropy strength. Our scheme is related to the {\it Asymptotic
  Preserving} method introduced in \cite{ShiJin}. These techniques are
designed to give a precise solution in the various regimes with no
restrictions on the computational meshes and with additional property
of converging to the limit solution when $\varepsilon \rightarrow
0$. The derivation of the Asymptotic Preserving method requires
identification of the limit model. In the case of Singular
Perturbation problems, the original problem is reformulated in such a
way that the obtained set of equations contain both the limit model
and the original problem with a continuous transition between them,
according to the values of $\varepsilon$. This reformulated system of
equation sets the foundation of the AP-scheme. These Asymptotic
Preserving techniques have been explored in previous studies, for instance
quasi-neutral or gyro-fluid limits \cite{Crispel,Sangam}, as well as
anisotropic elliptic problems of the form (\ref{eq:Jg0a}) with vector $b$ aligned
with a coordinate axis \cite{DDN,besse}.

In this paper, we present a new algorithm which extends the results of
\cite{DDN}. The originality of this algorithm consists in the fact,
that it is applicable for variable anisotropy directions $b$, without
additional work. The discretization mesh has not to be adapted to the
field direction $b$, but is simply a Cartesian grid, whose mesh-size
is governed by the desired accuracy, independently on the anisotropy
strength $\eps$. All this is possible by a well-adapted mathematical
framework (optimally chosen spaces, introduction of Lagrange
multipliers). The key idea, as in \cite{DDN}, is to decompose the
solution $\phi $ into two parts: a mean part $p$ which is constant
along the field lines and the fluctuation part $q$ consisting of a
correction to the mean part needed to recover the full solution. Both
parts $p$ and $q$ are solutions to well-posed problems for any
$\varepsilon>0 $. In the limit of $\varepsilon \rightarrow 0 $ the
AP-reformulation reduces to the so called {\it Limit} model (L-model),
whose solution is an acceptable approximation of the P-model solution
for $\varepsilon \ll 1 $ (see Theorem \ref{thm_EX}). In \cite{DDN} the
Asymptotic Preserving reformulation of the original problem was
obtained in two steps. Firstly, the original problem was integrated
along the field lines ($z$-axis) leading to an $\varepsilon
$-independent elliptic problem for the mean part $p$. Secondly, the
mean equation was subtracted from the original problem and the
$\varepsilon $-dependent elliptic problem for the fluctuating part $q$
was obtained. This approach however is not applicable if the field $b$
is arbitrary. In this paper we present a new approach. Instead of
integrating the original problem along the arbitrary field lines, we
choose to force the mean part $p$ to lie in the Hilbert space of
functions constant along the field lines and the fluctuating part $q$
to be orthogonal (in $L^2$ sense) to this space. This is done by a
Lagrange multiplier technique and requires introduction of additional
variables thus enlarging the linear system to be solved. This method
allows to treat the arbitrary $b$ field case, regardless of the field
topology and thus eliminates the
limitations of the algorithm presented in \cite{DDN}.
\textcolor{black}{We note that an alternative method, bypassing the need in Lagrange multipliers, is proposed in \cite{brull}. 
It is based on a reformulation of the original problem as a fourth order equation.}
\\

The outline of this paper is the following. Section \ref{sec:prob_def}
introduces the original anisotropic elliptic problem. The original
problem will be referred to as the Singular-{\it Perturbation} model
(P-model). The mathematical framework is introduced and the {\it
  Asymptotic Preserving} reformulation (AP-model) is then
derived. Section \ref{sec:num_met} is devoted to the numerical
implementation of the AP-formulation. Numerical results are presented
for 2D and 3D test cases, for constant and variable fields $b$. Three
methods are compared (AP-formulation, P-model and L-model) according
to their precision for different values of $\varepsilon $. The
rigorous numerical analysis of this new algorithm will be the subject
of a forthcoming publication.

%%%%%%%%%%%%%%%%%%%%%%%
\section{Problem definition}\label{sec:prob_def}
%%%%%%%%%%%%%%%%%%%%%%%

We consider a two or three dimensional anisotropic problem, given on a
sufficiently smooth, bounded domain $\Omega \subset \mathbb R^d$,
$d=2,3$ with boundary $\partial \Omega$. The direction of the
anisotropy is defined by the vector field $b \in
(C^{\infty}(\Omega))^d$, satisfying $|b(x)|=1$ for all $x \in \Omega$.

\noindent Given this vector field $b$, one can decompose now vectors
$v \in \mathbb R^d$, gradients $\nabla \phi$, with $\phi(x)$ a scalar
function, and divergences $\nabla \cdot v$, with $v(x)$ a vector
field, into a part parallel to the anisotropy direction and a part
perpendicular to it.  These parts are defined as follows:
\begin{equation} 
  \begin{array}{llll}
    \ds v_{||}:= (v \cdot b) b \,, & \ds v_{\perp}:= (Id- b \otimes b) v\,, &\textrm{such that}&\ds
    v=v_{||}+v_{\perp}\,,\\[3mm]
    \ds \nabla_{||} \phi:= (b \cdot \nabla \phi) b \,, & \ds
    \nabla_{\perp} \phi:= (Id- b \otimes b) \nabla \phi\,, &\textrm{such that}&\ds
    \nabla \phi=\nabla_{||}\phi+\nabla_{\perp}\phi\,,\\[3mm]
    \ds \nabla_{||} \cdot v:= \nabla \cdot v_{||}  \,, & \ds
    \nabla_{\perp} \cdot v:= \nabla \cdot v_{\perp}\,, &\textrm{such that}&\ds
    \nabla \cdot v=\nabla_{||}\cdot v+\nabla_{\perp}\cdot v\,,
  \end{array}
\end{equation}
where we denoted by $\otimes$ the vector tensor product. With these
notations we can now introduce the mathematical problem, the so-called
Singular Perturbation problem, whose numerical solution is the main
concern of this paper.

%%%%%%%%%%%%%%%%%%%%%%%
\subsection{The Singular Perturbation problem (P-model)}
%%%%%%%%%%%%%%%%%%%%%%%
We consider the following Singular Perturbation problem
\begin{gather}
    (P)\,\,\,
  \left\{
    \begin{array}{ll}
      -{1 \over \varepsilon} \nabla_\parallel \cdot 
      \left(A_\parallel \nabla_\parallel \phi^{\varepsilon }\right) 
      - \nabla_\perp \cdot 
      \left(A_\perp \nabla_\perp \phi^{\varepsilon }\right) 
      = f 
      & \text{ in } \Omega, \\[3mm]
      {1 \over \varepsilon} 
      n_\parallel \cdot 
      \left( A_\parallel \nabla_\parallel \phi^{\varepsilon } \right)
      +
      n_\perp \cdot 
      \left(A_\perp \nabla_\perp \phi^{\varepsilon }\right) 
      = 0
      & \text{ on } \partial\Omega _N,  \\[3mm]
      \phi^{\varepsilon }= 0
      & \text{ on } \partial\Omega _D\,,
    \end{array}
  \right.
  \label{eq:J07a} 
\end{gather}
where $n$ is the outward normal to $\Omega $ and the boundaries are defined by 
\begin{gather}
  \partial\Omega _D = \{ x \in \partial\Omega \ | \ b (x) \cdot n = 0
  \},\quad \quad \partial\Omega_N = \partial\Omega
  \setminus \partial\Omega_D
  \label{eq:Ju9a}.
\end{gather}
The parameter $0<\eps <1$ can be very small and is responsible for the
high anisotropy of the problem. The aim is to introduce a numerical
scheme, whose computational costs (simulation time and memory), for
fixed precision, are independent of $\eps$.\\
We shall assume in the rest of this
paper the following hypothesis on the diffusion coefficients and the
source terms\\

\noindent {\bf Hypothesis A} \label{hypo} {\it
  Let $f \in L^2(\Omega)$ and $\overset{\circ}{\partial\Omega _D} \neq \varnothing$.
  The diffusion coefficients $A_{\parallel} \in
  L^{\infty} (\Omega)$ and $A_{\perp} \in \MM_{d \times d} (L^{\infty}
  (\Omega))$ are supposed to satisfy
  \begin{gather}
        0<A_0 \le A_{\parallel}(x) \le A_1\,, \quad  \textrm{f.a.a.}\,\,\,x \in \Omega,
        \label{eq:J48a1}
        \\[3mm]
        A_0 ||v||^2 \le v^t A_{\perp}(x) v \le A_1 ||v||^2\,,
        \quad \forall v\in \mathbb R^d\,\,\, \text{and} \,\,\,  \textrm{f.a.a.}\,\,\, x \in \Omega.
        \label{eq:J48a3}
  \end{gather}
}

\noindent As we intend to use the finite element method for the
numerical solution of the P-problem, let us put (\ref{eq:J07a}) under
variational form. For this let ${\cal V}$ be the Hilbert space
$$
{\cal V}:=\{ \phi \in H^1(\Omega)\,\, / \,\, \phi_{| \partial \Omega_D} =0 \}
\,, \quad (\phi,\psi)_{\cal V}:= (\nabla_{||} \phi,\nabla_{||}
\psi)_{L^2} + \eps (\nabla_{\perp} \phi,\nabla_{\perp}
\psi)_{L^2}\,.
$$
Thus, we are seaking for $\phi^\eps \in {\cal V}$, the solution of
\be \label{eq:Ja8a}
a_{||}(\phi^\eps,\psi) + \eps a_{\perp}(\phi^\eps,\psi)=\eps (f,\psi)\,, \quad
\forall \psi \in {\cal V}\,,
\ee
where $(\cdot,\cdot)$ stands for the standard $L^2$ inner product and the continuous bilinear forms $a_{||} : \cal{V} \times \cal{V} \rightarrow
\RR$ and $a_{\perp}: \cal{V} \times \cal{V} \rightarrow \RR$ are given by
\be \label{bil}
\begin{array}{lll}
  \ds a_{||}(\phi,\psi)&:=&\ds \int_{\Omega} A_{||} \nabla_{||}
  \phi \cdot \nabla_{||}\psi\, dx\,, \quad  a_{\perp}(\phi,\psi):=\ds \int_{\Omega} ( A_{\perp} \nabla_{\perp}
  \phi) \cdot \nabla_{\perp}\psi\, dx\,.
\end{array}
\ee
Thanks to Hypothesis A and the Lax-Milgram theorem, problem
(\ref{eq:J07a}) admits a unique solution $\phi^\eps \in {\cal V}$ for
all fixed $\eps >0$.

%%%%%%%%%%%%%%%%%%%%%%%
\subsection{The Limit problem (L-model)}
%%%%%%%%%%%%%%%%%%%%%%%
The direct numerical solution of (\ref{eq:J07a}) may be very
inaccurate for $\varepsilon \ll 1$. Indeed, when $\varepsilon$ tends
to zero, the system reduces to
\begin{gather}
  \left\{
    \begin{array}{ll}
      \displaystyle
      -\nabla_\parallel \cdot 
      \left(A_\parallel \nabla_\parallel \phi\right) 
      = 0 
      & \text{ in } \Omega, \\[3mm]
      \displaystyle
      n_\parallel \cdot 
      \left( A_\parallel \nabla_\parallel \phi \right)
      = 0
      & \text{ on } \partial\Omega _N,  \\[3mm]
      \displaystyle
      \phi= 0
      & \text{ on } \partial\Omega _D.
    \end{array}
  \right.
  \label{eq:Jc8a}
\end{gather}
This is an ill-posed problem as it has an infinite number of solutions
$\phi \in \mathcal G$, where
\begin{gather}
  \mathcal G = \{ \phi \in \mathcal V \ | \ \nabla_\parallel \phi =0\}\,,
  \label{eq:Jd8a}
\end{gather}
is the Hilbert space of functions, which are constant along the field
lines of $b$. This shows that the condition number of the system
obtained by discretizing (\ref{eq:J07a}) tends to $\infty$ as
$\varepsilon\to 0$ so that its solution will suffer from round-off
errors.

For this reason, we should approximate (\ref{eq:J07a}) in the limit
$\varepsilon \rightarrow 0$ differently. Supposing that
$\phi^{\varepsilon } \rightarrow \phi^{0}$ as $\varepsilon \rightarrow
0$ we identify (at first formally) the problem satisfied by
$\phi^{0}$. From the above arguments we know that $\phi^{0} \in
\mathcal G$. Taking now test functions $\psi \in \mathcal G$ in
(\ref{eq:Ja8a}), we obtain
\begin{gather}
  \int_\Omega  A_{\perp} \nabla_{\perp} \phi^{\varepsilon }
  \cdot 
  \nabla_{\perp} \psi \, dx
  =
  \int_\Omega f \psi \, dx
  \label{eq:Jn8a}
  .
\end{gather}
Passing to the limit $\varepsilon \rightarrow 0$ into this equation yields the variational
formulation of the problem satisfied by $\phi^{0}$ (Limit problem):
find $\phi^{0} \in \mathcal G$, the solution of
\begin{gather}
  (L)\,\,\,
  \int_\Omega  A_{\perp} \nabla_{\perp} \phi^{0}
  \cdot 
  \nabla_{\perp} \psi \, dx
  =
  \int_\Omega f \psi \, dx
  \;\; , \;\;  \forall \psi \in \mathcal G\,,
  \label{eq:Jv9a}
\end{gather}
which is a well posed problem. Indeed, the space ${\cal G} \subset {\cal V}$ is a Hilbert space, associated with the inner product
\be \label{sc_G}
(\phi,\psi)_{\cal G}:= (\nabla_{\perp} \phi,\nabla_{\perp}
\psi)_{L^2}\,, \quad \forall \phi, \psi \in {\cal G}\,,
\ee
and the norm $||\cdot ||_{\cal G}$ is equivalent to the $H^1$ norm. This is due to the Poincar\'e inequality, as
$$
||\phi||_{L^2}^2 \le C ||\nabla \phi||_{L^2}^2 = C ||\nabla_{||}
\phi||_{L^2}^2+C ||\nabla_{\perp} \phi||_{L^2}^2 = C ||\nabla_{\perp}
\phi||_{L^2}^2\,, \quad \forall \phi \in {\cal G}\,.
$$
Hypothesis A and the Lax-Milgram lemma
imply the existence and uniqueness of a solution $\phi^0 \in
{\cal G}$ of the Limit problem (\ref{eq:Jv9a}).

\bigskip 
\color{black}
\begin{rem}\label{remark:limit:UniformB}
%%%%%%%%%%%%%%%%%%%%%%
Let us restrict ourselves for the moment to the simple special case (considered in a previous paper \cite{DDN}) of the two dimensional domain $\Omega  =
  (0,L_x)\times(0,L_z)$ in the $(x,z)$  plane with a constant $b$-field aligned with the $Z$-axis:
\begin{gather}
  b= 
  \left(
    \begin{array}{c}
      0 \\
      1
    \end{array}
  \right)
  \label{eq:Jy9a}
  .
\end{gather}
  The functions in the space ${\cal G}$ are
  independent of $z$ so that ${\cal G}$
  can be identified to $H^1_0(0,L_x)$. The limit
  problem \eqref{eq:Jv9a} now reads:     Find $\phi^0$ in
  $H^1_0(0,L_x)$ verifying
  \begin{equation*}
    \int_0^{L_x} \bar A_\perp(x) \partial_x \phi^0(x) \, \partial_x
    \psi(x) \, dx
    = \int_0^{L_x} \bar f(x) \psi(x) \, dx \,, \qquad \forall \psi \in H^1_0(0,L_x)\,,\\
  \end{equation*}
where $ \bar A_\perp(x) = (1/L_z) \int_0^{\fabrice{L_z}}
A_{\perp,11}(x,z) \, dz$ and $ \bar f(x) = (1/L_z) \int_0^{\fabrice{L_z}} f(x,z) \, dz$ are the mean values of $A_\perp$ and $f$ along the field lines. 
The limit solution $\phi^0$ thus verifies a one-dimensional elliptic equation whose
coefficients are integrated along the anisotropy direction:
\begin{equation}\label{eq:limit:uniformB}
  \begin{split}
  & - \partial_x \Big(\bar A_\perp(x) \, \partial_x \phi^0(x) \Big) = \bar
  f(x)  \text{ on } (0,L_x),   \\
  & \phi^0(0) = \phi^0(L_x) = 0 \,.
  \end{split}
\end{equation}

We see now that $\phi^0(x)$ is a solution to the one dimensional elliptic problem so that it belongs to $H^2(0,L_x)$ provided $f\in L^2(\Omega)$. 
Since $\phi^0$ as a function of $(x,z)$ does not depend on $z$, we have also $\phi^0\in H^2(\Omega)$. This conclusion ($\phi^0\in H^2(\Omega)$) remains valid in the case of 
a cylindrical three dimensional domain $\Omega  = \Omega_{xy}\times(0,L_z)$ in the $(x,y,z)$ space with any sufficiently smooth $\Omega_{xy}$ in the $(x,y)$  plane 
and the field $b$ aligned with the $Z$-axis, $b=(0,0,1)^t$. Indeed, it is easy to see that $\phi^0=\phi^0(x,y)$ solves in this case an elliptic two dimensional problem
in $\Omega_{xy}$ similar to (\ref{eq:limit:uniformB}) so that we can apply the standard regularity results for the elliptic problems. 
These examples show that it is reasonable to suppose $\phi^0\in H^2(\Omega)$ also in more general geometries of $\Omega$ and $b$. This can be indeed proved under the hypotheses 
in Appendix \ref{appA} by specifying the $(d-1)$ dimensional elliptic problem for $\phi^0$. The proof being rather lengthy and technical, we prefer to postpone it to a forthcoming work \cite{AJC}.
\end{rem}
\color{black}

%%%%%%%%%%%%%%%%%%%%%%%
\subsection{The Asymptotic Preserving approach (AP-model)}
%%%%%%%%%%%%%%%%%%%%%%%
In this section we introduce the AP-formulation, which is a
reformulation of the Singular Perturbation problem (\ref{eq:J07a}),
permitting a ``continuous'' transition from the (P)-problem
(\ref{eq:J07a}) to the (L)-problem (\ref{eq:Jv9a}), as $\eps
\rightarrow 0$. For this purpose, each function is decomposed into its
mean part along the anisotropy direction (lying in the subspace
$\mathcal{G}$ of $\mathcal{V}$) and a fluctuating part
(cf. \cite{DDN}) lying in the $L^2$-orthogonal complement
$\mathcal{A}$ of $\mathcal{G}$ in $\mathcal{V}$, defined by
\begin{gather}
  \mathcal A : = 
   \{ \phi \in \mathcal V \ | (\phi,\psi ) =0 \;\; , \;\;  \forall
  \psi \in \mathcal G\}\,.
  \label{eq:Jg8a}
\end{gather}
Note that $(\cdot,\cdot)$ denote here and elsewhere the inner product
of $L^2(\Omega)$.

In what follows, we need the following\\

\noindent {\bf Hypothesis B} {\it The Hilbert-space $\mathcal V$ admits the
  decomposition
\begin{gather}
  \mathcal V = \mathcal G \oplus^{\perp} \mathcal A
\label{eq:Jf8a},
\end{gather}
with ${\cal G}$ given by (\ref{eq:Jd8a}) and ${\cal A}$ given by
(\ref{eq:Jg8a}) and where the orthogonality of the direct sum is taken
with respect to the $L^2$-norm. Denoting by $P$ the orthogonal
projection on $\mathcal G$ with respect to the $L^{2}$ inner product:
\begin{gather}
  P : \mathcal V \rightarrow \mathcal G\,\,\text{ such that }\,\,
  (P\phi,\psi)=(\phi,\psi)\ \ \forall\phi\in\mathcal V,\, \psi\in\mathcal G\,,
  \label{eq:Je8a}
\end{gather}
we shall suppose that this mapping is continuous and that we have the
Poincar\'e-Wirtinger inequality
\begin{equation}
||\phi -P\phi ||_{L^{2}(\Omega )}\leq C||\nabla _{||}\phi ||_{L^{2}(\Omega
)}\,,\quad \forall \phi \in \mathcal{V}\,.  \label{PoinW}
\end{equation}
}

\noindent Applying the projection $P$ to a function $\phi$ is nothing but a
weighted average of $\phi$ along the anisotropy field lines of
$b$. The space ${\cal G}$ is the space of averaged functions (the
parallel $\cal{G}$radient of these averaged functions being equal to
zero), whereas the space ${\cal A}$ is the space of the fluctuations
(the $\cal{A}$verage of the fluctuations being equal to zero).  Note
that the decomposition (\ref{eq:Jf8a}) is not self evident and it may
in fact fail on some ``pathological'' domains $\Omega$. Indeed,
although one can always define an $L^2$-orthogonal projection $\tilde
P\phi$ on the space of functions constant along each field line, for
any $\phi$ with square-integrable $\nabla_{||}\phi$, one cannot assure
in general that $\tilde P\phi$ belongs to $\mathcal{V}$ for
$\phi\in\mathcal{V}$ since one may lose control of the perpendicular
part of the gradient of $\tilde P\phi$.  Fortunately however,
Hypothesis B is typically satisfied for the domains of practical
interest.  The interested reader is referred to Appendix \ref{appA}
for an example of a set of assumptions on $\Omega$ and $b$ which
entail Hypothesis B and which resume essentially to the requirement
for the field $b$ to intersect $\partial\Omega_N$ in a uniformly
non-tangential manner and for the boundary components
$\partial\Omega_N$ and $\partial\Omega_D$ to be sufficiently smooth.

Let us also define the operator 
\begin{gather}
  Q : \mathcal V \rightarrow \mathcal A\,, \quad Q = I - P\,.
  \label{eq:Jh8a}
\end{gather}
Each function $\phi \in \mathcal V$ can be decomposed uniquely as
$\phi = p + q$, where $p = P\phi \in \mathcal G$ and $q = Q\phi \in
\mathcal A$. Using this decomposition, we reformulate the Singular-Perturbation problem (\ref{eq:J07a}). Indeed, replacing $\phi^{\varepsilon}:=p^{\varepsilon}+ q^{\varepsilon}$
in problem (\ref{eq:J07a}) and taking test functions $\eta \in
\mathcal G$ and $\xi \in \mathcal A$ leads to an asymptotic
preserving formulation of the original problem: Find $(p^\varepsilon
,q^\varepsilon ) \in \mathcal G \times \mathcal A$ such that
\begin{gather}
  \left\{
    \begin{array}{ll}
      \displaystyle
      a_{\perp} (p^{\varepsilon },\eta ) + a_{\perp} (q^{\varepsilon},\eta ) = (f,\eta )
      & \forall \eta \in \mathcal G, \\[3mm]
      \displaystyle
      a_{||} (q^{\varepsilon },\xi ) + \varepsilon a_{\perp}
      (q^{\varepsilon},\xi) + \varepsilon a_{\perp} (p^{\varepsilon},\xi)
      = \varepsilon (f, \xi )
      & \forall \xi \in \mathcal A.
    \end{array}
  \right.
  \label{eq:Ji8a}
\end{gather}
Contrary to the Singular Perturbation problem (\ref{eq:J07a}), setting
formally $\eps=0$ in (\ref{eq:Ji8a}) yields the system
\be \label{sy}
\left\{ 
\begin{array}{lll}
\ds a_{\perp}(p^0,\eta)+a_{\perp}(q^0,\eta) &=&\ds(f,\eta) \,, \quad \forall
\eta \in {\cal G}\\[3mm]
\ds a_{||}(q^0,\xi )
&=&\ds 0 \,, \quad \forall \xi \in {\cal A}\,,
\end{array}
\right.
\ee
which has a unique solution $(p^0,q^0) \in \cal{G} \times \cal{A}$,
where $p^0$ is the unique solution of the L-problem (\ref{eq:Jv9a}) and
$q^0 \equiv 0$. Indeed, taking $\xi=q^0$ as test function in the second equation of (\ref{sy}) yields $\nabla_{||} q^0 =0$, which means $q^0 \in \cal{G}$. 
But at the same time, $q^0 \in \cal{A}$, so that $q^0 \in {\cal G} \cap {\cal A} = \{ 0 \}$. Setting then $q^0 \equiv 0$ in the first equation of (\ref{sy}), shows that $p^0$ is the unique solution of the L-problem. 

\begin{thm}\label{thm_EX} 
  For every $\eps>0$ the Asymptotic Preserving formulation (\ref%
  {eq:Ji8a}), under Hypotheses A and B, admits a unique solution
  $(p^{\eps},q^{\eps%
  })\in \mathcal{G}\times \mathcal{A}$, where $\phi
  ^{\eps}:=p^{\eps}+q^{\eps}$ is the unique solution in $\mathcal{V}$
  of the Singular Perturbation model (%
  \ref{eq:J07a}).\newline These solutions satisfy the bounds
  \begin{equation}
    ||\phi ^{\eps}||_{H^{1}(\Omega )}\leq C||f||_{L^{2}(\Omega )}\,,\quad ||q^{%
      \eps}||_{H^{1}(\Omega )}\leq C||f||_{L^{2}(\Omega )}\,,\quad ||p^{\eps%
    }||_{H^{1}(\Omega )}\leq C||f||_{L^{2}(\Omega )}\,,  \label{est_sol}
  \end{equation}%
  with an $\eps$-independent constant $C>0$. Moreover, we have
  \begin{equation}
    \phi^{\eps}\rightarrow \phi^{0},\,\ 
    p^{\eps}\rightarrow \phi^{0}\,\ \text{and}\quad 
    q^{\eps}\rightarrow
    0\quad \text{in}\quad H^{1}(\Omega )\text{ as }\eps\rightarrow 0\,,  \label{wconv_sol}
  \end{equation}%
  where $\phi^{0}\in \mathcal{G}$ is the unique solution of the Limit model (\ref{eq:Jv9a}).
\end{thm}

\begin{proof}
  The existence and uniqueness of a solution for the P-problem as well
  as L-problem are consequences of the Lax-Milgram theorem. The
  existence and uniqueness of a solution of (\ref{eq:Ji8a}) is then
  immediate by construction, remarking that the decomposition $\phi
  ^{\eps}=p^{\eps}+q^{\eps%
  }$ is unique.\newline The bound $||\phi ^{\eps}||_{H^{1}(\Omega
    )}\leq C||f||_{L^{2}(\Omega )}$ is obtained by a standard elliptic
  argument. Furthermore, $p^{\eps}=P\phi ^{\eps%
  }$ where $P$ is the $L^{2}$-orthogonal projector on $\mathcal{G}$,
  which is a bounded operator in $\mathcal{V}$ by (\ref{reg}). This
  implies the estimates for $p^{\eps}$ and $q^{\eps}$ in
  (\ref{est_sol}). Since $p^{\eps} \in \mathcal{G}$ and $q^{\eps} \in
  \mathcal{A}$ are bounded, there exist subsequences $p^{\eps_{n}}$
  and $q^{\eps_{n}}$ that weakly converge for $\varepsilon
  _{n}\rightarrow 0$ to some $p^{0}\in \mathcal{G}$ and $q^{0}\in
  \mathcal{A}$. Taking $%
  \varepsilon =\varepsilon _{n}$ in (\ref{eq:Ji8a}) and passing to the
  limit $%
  \varepsilon _{n}\rightarrow 0$ we identify $(p^{0},q^{0})$ with the
  unique solution of (\ref{sy}), i.e. $p^{0}=\phi^0$ is the unique solution
  of (\ref{eq:Jv9a}%
  ) and $q^{0}\equiv 0$. Since the limit does not depend on the choice
  of the subsequence, we have the weak convergence as $\varepsilon
  \rightarrow 0$, i.e.
\begin{equation*}
p^{\eps}\rightharpoonup _{\eps\rightarrow 0}p^{0}\quad \text{in}\quad
H^{1}(\Omega )\,,\quad q^{\eps}\rightharpoonup _{\eps\rightarrow 0}0\quad 
\text{in}\quad H{^{1}(\Omega )}\,.
\end{equation*}%
We shall prove now that these convergences are actually strong. Introducing $%
e^{\varepsilon }=p^{\varepsilon }-p^{0}$, we have%
\begin{equation*}
a_{\perp}(e^{\varepsilon },\eta )+a_{\perp}(q^{\varepsilon },\eta )=0,~\forall \eta
\in \mathcal{G}\,.
\end{equation*}%
Taking now $\eta =e^{\varepsilon }$ in this relation and adding it to the second
equation in (\ref{eq:Ji8a}), where we put $\xi =q^{\varepsilon }/\varepsilon $, yields%
\begin{equation}
\frac{1}{\varepsilon }a_{||}(q^{\varepsilon },q^{\varepsilon
})+a_{\perp}(q^{\varepsilon }+e^{\varepsilon },q^{\varepsilon }+e^{\varepsilon
})=(f,q^{\varepsilon })-a_{\perp}(p^{0},q^{\varepsilon }).  \label{aa1}
\end{equation}%
Due to the Poincar\'e-Wirtinger equation (\ref{PoinW}), there exist a constant $C>0$ such that
\begin{equation}\label{wirta}
||q||_{L^{2}(\Omega )}\leq Ca_{||}(q,q)^{1/2}\,,\quad \forall q\in \mathcal{A}\,.
\end{equation}%
In combination with a Young inequality this gives $(f,q^{\varepsilon })\leq
||f||_{L^{2}}||q^{\varepsilon }||_{L^{2}}\leq \varepsilon \frac{C^{2}}{2}%
||f||_{L^{2}}^{2}+\frac{1}{2\varepsilon }a_{||}(q^{\epsilon },q^{\epsilon })$%
. Using this in the right hand side of (\ref{aa1}), we arrive at 
\begin{equation*}
\frac{1}{2\varepsilon }a_{||}(q^{\varepsilon },q^{\varepsilon
})+a_{\perp}(q^{\varepsilon }+e^{\varepsilon },q^{\varepsilon }+e^{\varepsilon
})\leq \varepsilon \frac{C^{2}}{2}||f||_{L^{2}}^{2}
-a_{\perp}(p^{0},q^{\varepsilon }).
\end{equation*}

Noting that $q^{\varepsilon }+e^{\varepsilon }=\phi ^{\varepsilon }-p^{0}$
and $\nabla _{\Vert }e^{\varepsilon }=0$ we can rewrite this last inequality
as 
\begin{equation*}
\frac{1}{2\varepsilon }a_{||}(\phi ^{\varepsilon }-p^{0},\phi ^{\varepsilon
}-p^{0})+a_{\perp}(\phi ^{\varepsilon }-p^{0},\phi ^{\varepsilon }-p^{0})\leq
\varepsilon \frac{C^{2}}{2}||f||_{L^{2}}^{2} -a_{\perp}(p^{0},q^{\varepsilon }).
\end{equation*}%
Since $a_{\perp}(p^{0},q^{\varepsilon })\rightarrow 0$ as $\varepsilon
\rightarrow 0$ (thanks to the weak convergence $q^{\eps}\rightharpoonup 0$) we
observe that $\phi ^{\varepsilon }\rightarrow p^{0}$ strongly in $%
H^{1}(\Omega )$. Reminding again that $p^{\varepsilon }=P\phi ^{\varepsilon }
$ and $P$ is bounded in the norm of $H^{1}(\Omega )$, we obtain also $%
p^{\varepsilon }\rightarrow Pp^{0}=p^{0}$, which entails $q^{\varepsilon
}\rightarrow 0$. 
  %\hfill \qed
\end{proof}

%%%%%%%%%%%%%%%%%%%%%%%
%\subsection{Special case : constant vector field $b$ }
%\subsection{Special case : uniform and aligned $b$-field }
\bigskip 
\color{black}
\begin{rem}\label{remark:UniformB}
%%%%%%%%%%%%%%%%%%%%%%
Let us return to the simple special case discussed in remark~\ref{remark:limit:UniformB}, i.e. $\Omega =(0,L_x) \times
(0,L_z)$ and the $b$-field given by (\ref{eq:Jy9a}).
Remind that the space $\mathcal G$ can be identified in this case with the space of 
functions constant along the $Z$-axis, which means $\mathcal G := \{ \phi \in
{\cal V} \,\, / \,\, \partial_z \phi =0 \}$. The space $\mathcal A$ is
orthogonal (with respect to the $L^2$-norm) to $\mathcal G$ and thus
contains the functions that have zero mean value along the $Z$-axis,
i.e. $\mathcal A := \{ \phi \in {\cal V} \,\, / \,\, \int_0^{L_z}
\phi(x,z)\, dz =0\}$. Therefore, for $\phi^\eps = p^\eps + q^\eps \in
{\cal V}$, the function $p^{\varepsilon }$ is the mean value of
$\phi^{\varepsilon }$ in the direction of the field $b$:
\begin{gather}
  p^{\varepsilon } = \frac{1}{L_z}\int_0^{L_z} \phi^{\varepsilon } dz\,,
  \label{eq:Jz9a}
\end{gather}
and $q^{\varepsilon }$ is the fluctuating part with zero mean
value:
\begin{gather}
  q^{\varepsilon } = \phi^{\varepsilon } - \frac{1}{L_z}\int_0^{L_z} \phi^{\varepsilon } dz
  \label{eq:J29a}
  .
\end{gather}
Hypothesis B is thus easily verified.  The results obtained in this
special case were presented in a previous paper \cite{DDN}. In the
case of an arbitrary $b$-field, formula (\ref{eq:Jz9a}) is generalized
as (\ref{pdef}) in Appendix \ref{appA}, where the length element along
the $b$-field line is weighted by the infinitesimal cross-sectional
area of the field tube around the considered $b$-field-line. This
formula can be thus interpreted as a consequence of the co-area
formula. Note that in the special case of a uniform anisotropy
direction, the limit problem can easily be formulated as an elliptic
problem depending on the only transverse coordinates (see
equation~\eqref{eq:limit:uniformB}). The size of the 
problem is thus significantly smaller than that of the initial
one. This feature still occurs for non-uniform $b$-fields as long as
adapted coordinates and meshes are used. In our case, aligned and
transverse coordinates are not at our disposal and the solution of the
limit problem must be searched as a function of the whole set of
coordinates.
\end{rem}
\color{black}

%%%%%%%%%%%%%%%%%%%%%%%
\subsection{Lagrange multiplier space }
%%%%%%%%%%%%%%%%%%%%%%
The objective of this work is the numerical solution of system
(\ref{eq:Ji8a}) and the comparison of the obtained results with those
obtained by directly solving the original problem (\ref{eq:J07a}).  In
a general case, when the field $b$ is not necessarily constant, the
discretization of the subspaces $\mathcal G$ and $\mathcal A$, is not
straightforward, as in the simpler case \cite{DDN}. In order to
overcome this difficulty a Lagrange multiplier technique will be used.

%%%%%%%%%%%%%%%%%%%%%%%
\subsubsection{The $\mathcal A$ space}
%%%%%%%%%%%%%%%%%%%%%%%
To avoid the use of the constrained space $\mathcal A$, we can remark
that $\cal{A}$ can be characterized as being the orthogonal complement
(in the $L^2$ sense) of the $\cal{G}$-space. Thus, instead of
(\ref{eq:Ji8a}), the slightly changed system will be solved: find
$(p^{\varepsilon }, q^{\varepsilon }, l^{\varepsilon }) \in \mathcal G
\times \mathcal V \times \mathcal G$ such that
\begin{gather}
  \left\{
    \begin{array}{ll}
      \displaystyle
      a_{\perp} (p^{\varepsilon },\eta ) + a_{\perp} (q^{\varepsilon},\eta )
      = (f,\eta )
      & \forall \eta \in \mathcal G, \\[3mm]
      \displaystyle
      a_{||} (q^{\varepsilon },\xi ) + \varepsilon a_{\perp}
      (q^{\varepsilon},\xi) + \varepsilon a_{\perp} (p^{\varepsilon},\xi)
      + \left( l^{\varepsilon } , \xi \right)
      = \varepsilon (f, \xi ) 
      & \forall \xi \in \mathcal V , \\[3mm]
      \displaystyle
      \left( q^{\varepsilon } , \chi \right) =0
      & \forall \chi \in \mathcal G.
    \end{array}
  \right.
  \label{eq:Jj8a}
\end{gather}
The constraint $( q^{\varepsilon } , \chi ) =0$, $\forall \chi \in
\mathcal G$ is forcing the solution $q^{\varepsilon }$ to belong to
$\mathcal A$, and this property is carried over to the limit $\eps \rightarrow 0$. We have thus
circumvented the difficulty of discretizing $\mathcal A$ by introducing
a new variable and enlarging the linear system.

\begin{prop}\label{lem:AP-aquiv1}
  Problems (\ref{eq:Ji8a}) and (\ref{eq:Jj8a}) are
  equivalent. Indeed, $(p^\eps,q^\eps) \in \mathcal{G} \times
  \mathcal{A}$ is the unique solution of (\ref{eq:Ji8a}) if and only
  if $(p^\eps,q^\eps,l^\eps) \in {\cal G} \times {\cal V} \times {\cal
    G}$ with $l^\eps \equiv 0$ is the unique solution of
  (\ref{eq:Jj8a}).
\end{prop}

\begin{proof}
  Let $(p^\eps,q^\eps) \in \mathcal{G} \times \mathcal{A}$ be the
  unique solution of (\ref{eq:Ji8a}). Then, it is immediate to show
  that $(p^\eps,q^\eps,0)$ solves (\ref{eq:Jj8a}). Let now
  $(p^\eps,q^\eps,l^\eps) \in {\cal G} \times {\cal V} \times {\cal
    G}$ be a solution of (\ref{eq:Jj8a}). Then, the last equation of
  (\ref{eq:Jj8a}) implies that $q^\eps \in \mathcal A$. Choosing in
  the second equation as test function $\xi \in \mathcal G$, one gets
$$
\varepsilon a_{\perp}
      (q^{\varepsilon},\xi) + \varepsilon a_{\perp} (p^{\varepsilon},\xi)
      + \left( l^{\varepsilon } , \xi \right)
      = \varepsilon (f, \xi )\,, \quad\forall \xi \in \mathcal G\,,
$$
which because of the first equation in (\ref{eq:Jj8a}), yields $\left( l^{\varepsilon } , \xi
\right) = 0$ for all $\xi \in \mathcal G$. Thus $l^{\eps} \equiv 0$.
\end{proof}

%%%%%%%%%%%%%%%%%%%%%%%
\subsubsection{The $\mathcal G$ space}
%%%%%%%%%%%%%%%%%%%%%%%
In order to eliminate the problems that arise when dealing with the
discretization of $\mathcal G$, the Lagrange multiplier method will
again be used. First note that
\begin{gather}
  p \in \mathcal G 
  \Leftrightarrow
  \left\{
    \begin{array}{l}
      \nabla_{||} p = 0 \\[3mm]
      p \in \mathcal V
    \end{array}
  \right.
  \;
  \Leftrightarrow
  \;
  \left\{
    \begin{array}{l}
      \ds   \int_\Omega  A_{||} \nabla_{||} p \cdot \nabla_{||} \lambda
      \, dx = a_{||}(p,\lambda)= 0, \;\;  \forall 
      \lambda \in {\cal L} \\[3mm]
      p \in \mathcal V\,,
    \end{array}
  \right.
  \label{caract}
\end{gather}
where ${\cal L}$ is a functional space that should be chosen large
enough so that one could find for any $p\in{\cal V}$ a $\lambda\in{\cal L}$ with $\nabla_{||}
\lambda=\nabla_{||} p$. On the other hand, the
space ${\cal L}$ should be not too large in order to ensure the
uniqueness of the Lagrange multipliers in the unconstrained system. A
space that satisfies these two requirements under some quite general
assumptions to be detailed later, can be defined as \be\label{Lsp}
{\cal L} := \{ \lambda \in L^2(\Omega)\,\, / \,\, \nabla_{||} \lambda
\in L^2(\Omega)\,, \,\,\, \lambda_{| \partial \Omega_{in}} =0 \} \,,
\quad \textrm{with} \quad \partial \Omega_{in} := \{ x \in \partial
\Omega \,\, / \,\, b(x) \cdot n <0\}\,. \ee

Using the characterization (\ref{caract}) of the constrained space $
\mathcal G $, we shall now reformulate the system (\ref{eq:Jj8a}) as
follows: Find
$(p^\varepsilon,\;\lambda^\varepsilon,\;q^\varepsilon,\;l^\varepsilon,\;\mu^\varepsilon)
\in \mathcal V\times \mathcal L\times \mathcal V\times \mathcal V
\times \mathcal L$ such that
\begin{gather}
  (AP)\,\,\,
  \left\{
    \begin{array}{l}
      \displaystyle
      a_{\perp} (p^\varepsilon , \eta ) +
      a_{\perp} (q^\varepsilon , \eta ) + a_{||}(\eta,\lambda^\varepsilon )
      = \left(f,\eta  \right) \,, \quad \forall \eta \in \mathcal V\,,
      \\[3mm]
      \displaystyle
a_{||}( p^{\varepsilon },\kappa)=0\,,\quad \forall \kappa \in \mathcal L\,, 
     \\[3mm]
      \displaystyle
      a_{||} (q^\varepsilon , \xi ) +
      \varepsilon a_{\perp} (q^\varepsilon , \xi ) +
      \varepsilon a_{\perp} (p^\varepsilon , \xi ) +
      \left( l^{\varepsilon } , \xi \right)
      = \varepsilon \left(f,\xi  \right)  \,, \quad \forall \xi \in \mathcal V\,,
      \\[3mm]
      \displaystyle
      \left( q^{\varepsilon } , \chi \right) +
 a_{||}(\chi, \mu^\varepsilon)=0\,, \quad \forall \chi \in \mathcal V\,,
     \\[3mm]
      \displaystyle
a_{||}(l^\varepsilon,\tau)=0\,, \quad \forall \tau \in \mathcal  L\,.
    \end{array}
  \right.
  \label{eq:Ju7a}
\end{gather}
The advantage of the above formulation, as compared to
(\ref{eq:Ji8a}), is that we only have to discretize the spaces
$\mathcal V$ and $\mathcal L$ (at the price of the introduction of
three additional variables), which is much easier than the
discretization of the constrained spaces $\mathcal G$ and $\mathcal
A$. More importantly, the dual formulation (\ref{eq:Ju7a}) does not
require any change of coordinates to express the fact that $p^\eps$ is
constant along the $b$-field lines and that $q^\eps$ averages to zero
along these lines. Therefore this formulation is particularly well
adapted to time-dependent $b$-fields, as it does not require any
operation which would have to be reinitiated as $b$ evolves. The
system (\ref{eq:Ju7a}) will be called the Asymptotic-Preserving
formulation in the sequel.

To analyse this Asymptotic-Preserving formulation, we need the
following\\

\color{black}
\noindent {\bf Hypothesis B'} {\it The trace
  $\lambda_{| \partial \Omega_{in}}$ is well defined for any
  $\lambda\in\tilde{\cal V}$ as an element of $L^2(\partial
  \Omega_{in})$, with continuous dependence of the trace norm in
  $L^2(\partial \Omega_{in})$ on $||\lambda||_{\tilde{\cal V}}$.  Moreover, the Hilbert space 
  \be \label{V_tilde}
  \tilde{\cal V}= \{ \phi \in L^2(\Omega) \,\, / \,\, \nabla_{||} \phi
  \in L^2(\Omega) \}\,, \quad (\phi,\psi)_{\tilde{\cal V}} :=
  (\phi,\psi)+(\nabla_{||} \phi,\nabla_{||} \psi)\,, 
  \ee 
  admits the decomposition 
  \be\label{eq:Jf8aa} \tilde{\cal V} = \tilde{\cal G}
  \oplus \mathcal L\,, 
  \ee 
  where $\tilde{\cal G}$ is given by
  \be \label{G_tilde} \tilde{\cal G}:=\{ \phi \in \tilde{\cal V} \,\,
  / \,\, \nabla_{||} \phi =0 \}\,, \ee and ${\cal L}$ is given by
  (\ref{Lsp}). The spaces $\tilde{\cal G}$ and $G=\tilde{\cal G}\cap {\cal V}$ are related in the following way: 
  if $g\in\tilde{\cal G}$ is such that $\int_{\partial \Omega_{in}}\eta gd\sigma=0$ for all $\eta\in{\cal G}$, then $g=0$.}\\
\color{black}

\noindent The decomposition (\ref{eq:Jf8aa}) is quite natural. It tells simply
that any function $\phi$ can be decomposed on each field line as a sum
of a function that vanishes at one given point on this line and a
constant (which is therefore the value of $\phi$ at this
point). Hypothesis B' will be thus normally satisfied in cases of
practical interest. For example, we prove in Appendix \ref{appA} that
the set of assumptions on the domain $\Omega$ and the $b$-field which
can be used to verify Hypothesis B, is also sufficient (but far from
necessary) for Hypothesis B'.  We are now able to show the relation
between systems (\ref{eq:Jj8a}) and (\ref{eq:Ju7a}).
\begin{prop}\label{lem:AP-aquiv2}
  Assuming Hypotheses A, B and B', problem (\ref{eq:Ju7a}) admits a
  unique solution\newline
  $(p^\varepsilon,\;\lambda^\varepsilon,\;q^\varepsilon,\;l^\varepsilon,\;\mu^\varepsilon)\in
  \mathcal V\times \mathcal L\times \mathcal V\times \mathcal V \times
  \mathcal L$, where $(p^{\varepsilon }, q^{\varepsilon
  },l^{\varepsilon })\in \mathcal G\times \mathcal V\times \mathcal G$
  is the unique solution of (\ref{eq:Jj8a}).
\end{prop}

The proof of Proposition \ref{lem:AP-aquiv2} is based on the following two lemmas
\begin{lem} \label{lemlem0} Assume Hypothesis B' and let $p \in
  \tilde{\cal V}$ be such that $a_{||}( p, \lambda)=0$,
  $\forall\lambda \in {\cal L}$.  Then $p \in \tilde{\cal G}$.
\end{lem} 
\begin{proof} Take any $\eta \in \tilde{\cal V}$ and decompose
  $\eta=\lambda +g$ with $\lambda \in {\cal L}$ and $g\in \tilde{\cal
    G}$.  We have $a_{||}( p, g)=0$, hence $a_{||}( p, \eta)=0$ for
  all $\eta \in \tilde{\cal V}$. This entails $\nabla_{||}p=0$, hence
  $p \in \tilde{\cal G}$.
\end{proof}
\begin{lem} \label{lemlem} Assume Hypothesis B' and let $F \in
  \tilde{\cal V}^*$ be such that $F( \eta)=0$ for all $\eta \in
  {\cal G}$.  Then the problem of finding $\lambda \in {\cal L}$
  such that \be\label{MvF} a_{||}( \eta, \lambda)=F(\eta)\,, \quad
  \forall \eta \in \tilde{\cal V}\,, \ee has a unique solution.
\end{lem}
\begin{proof}
  Consider the bilinear form $b$ on $\tilde{\cal V}\times\tilde{\cal V}$ 
  $$
  b(u,v)=a_{||}(u,v)+\int_{\partial \Omega_{in}}uvd\sigma
  $$
  By Hypothesis B', this is an inner product on $\tilde{\cal
    V}$. Indeed, if $b(u,u)=0$ then $u\in\tilde{\cal G} \cap {\cal L}$
  so that $u=0$. Riesz representation theorem implies that the problem
  of finding $\mu \in \tilde{\cal V}$ such that
  $$
  b( \eta, \mu)=F(\eta)\,, \quad \forall \eta \in \tilde{\cal V}\,,
  $$
  has a unique solution. We can now decompose $\mu=\lambda +g$ with
  $\lambda \in {\cal L}$ and $g\in \tilde{\cal G}$. This yields
  $$
  a_{||}(\eta,\lambda)+\int_{\partial \Omega_{in}}\eta gd\sigma=F(\eta)\,,
  \quad \forall \eta \in \tilde{\cal V}\,, 
  $$
  so that, in particular, $\int_{\partial \Omega_{in}}\eta gd\sigma=0$ for all $\eta\in{\cal G}$ which implies $g=0$. We see now that $\lambda$ is a
  solution to (\ref{MvF}). The uniqueness follows easily.
\end{proof}
\newline Let us now prove Proposition \ref{lem:AP-aquiv2}.\\
\textbf{Proof of existence in Proposition \ref{lem:AP-aquiv2}.}
Take $(p^{\varepsilon }, q^{\varepsilon},l^{\varepsilon })\in \mathcal
G\times \mathcal V\times \mathcal G$  as the unique solution of
(\ref{eq:Jj8a}). Then, equations 2,3,5 in (\ref{eq:Ju7a}) are
immediately satisfied. It remains to choose properly the
Lagrange multipliers ${\lambda}^\eps, {\mu}^\eps \in {\cal L}$  to satisfy equations 1,4  in (\ref{eq:Ju7a}). For 
this, let us define $F_1, F_2 \in \tilde{\cal V}^*$ by
\be \label{NR0}
F_1(\eta):= { 1 \over \eps} a_{||} (q^\eps,\eta)\,, \quad
F_2(\eta):=-(q^\eps,\eta)\,, \quad \forall \eta \in \tilde{\cal V}\,. 
\ee
These functionals are indeed continuous in the norm of $\tilde{\cal
  V}$ since their definitions do not contain the derivatives in directions perpendicular to $b$. 
Since $F_1(\eta)=F_2(\eta)=0$ for all $\eta \in{\cal G}$,  Lemma \ref{lemlem} implies the existence of
${\lambda}^\eps \in {\cal L}$ and ${\mu}^\eps \in {\cal L}$, such that
\be \label{NR}
a_{||}( \eta, {\lambda}^\eps) =F_1(\eta)\,, \quad  a_{||}(\chi,
{\mu}^\eps ) =F_2(\chi)\,, \quad \forall \eta, \chi \in \tilde{\cal
  V}\,. 
\ee
Taking $\eta,\chi\in{\cal V} \subset \tilde{\cal V}$ we observe (cf. the second line in (\ref{eq:Jj8a}) where $l^\eps=0$)
$$
a_{||}( \eta,{\lambda}^\eps)= { 1 \over \eps} a_{||} (q^\eps,\eta)=
(f,\eta) -a_{\perp}(p^\eps,\eta)-a_{\perp}(q^\eps,\eta)\,, \quad
\forall \eta \in {\cal V}\,, 
$$
$$
a_{||}(\chi, {\mu}^\eps) = -(q^\eps,\chi)\,, \quad \forall \chi \in {\cal V}\,,
$$
which coincides with equations 1,4 in (\ref{eq:Ju7a}).\\
\textbf{Proof of uniqueness in Proposition \ref{lem:AP-aquiv2}.}  Consider the solution to system
(\ref{eq:Ju7a}) with $f=0$. Lemma \ref{lemlem0} implies then that
$p^{\varepsilon }, l^{\varepsilon }\in \tilde{\mathcal{G}}\cap\mathcal{V}=\mathcal{G}$
and
$(p^{\varepsilon }, q^{\varepsilon},l^{\varepsilon })\in \mathcal
G\times \mathcal V\times \mathcal G$ verifies (\ref{eq:Jj8a}) with
$f=0$ so that $p^{\varepsilon } = q^{\varepsilon} = l^{\varepsilon } =
0$ by Proposition \ref{lem:AP-aquiv1}. Equations 1,4 in (\ref{eq:Ju7a}) now tell us that ${\lambda}^\eps,
{\mu}^\eps \in \tilde{\cal G}$, but $\tilde{\cal G} \cap {\cal L} =
\{0\}$, hence ${\lambda}^\eps = {\mu}^\eps =0$.
\hfill$\Box$

\color{black}

\bigskip
The presence of $1/\eps$ in the formulas (\ref{NR0}), (\ref{NR}) defining $\lambda^\eps$ indicates at a first sight that $\lambda^\eps$ may tend to $\infty$
as $\eps\to 0$ which would be disastrous for an AP numerical method based on (\ref{eq:Ju7a}) at very small $\eps$. Fortunately $\lambda^\eps$ remains bounded 
uniformly in $\eps$ in the cases of practical interest. It suffices to suppose that the limit solution $\phi^0$ is in $H^2(\Omega)$ which is a reasonable assumption 
as discussed in Remark \ref{remark:limit:UniformB}. 
\begin{prop}\label{lamBoundLem}
Assume Hypotheses A, B, B' and  $\phi^0 \in H^2(\Omega)$ where $\phi^0$ is the solution to (\ref{eq:Jv9a}). Then $\lambda^\eps$ introduced in (\ref{eq:Ju7a}) satisfies
\be\label{lamBound}
||\nabla_{||}\lambda^\eps||_{L^2} \le C\max(||f||_{L^2},||\phi^0||_{H^2})
\ee
with a constant $C$ independent of $\eps$.
\end{prop}
\begin{proof}
We will denote all the $\eps$-independent constants by $C$ in this proof. 
We start from relation (\ref{aa1}) in the proof of Theorem \ref{thm_EX}. Dropping the positive term $a_\perp(q^\eps+e^\eps,q^\eps+e^\eps)$ it can be rewritten as
$$
\frac 1 \eps a_{||}(q^\eps,q^\eps) \le (f,q^\eps) - a_\perp(\phi^0,q^\eps).
$$
Since $\phi^0 \in H^2(\Omega)$ we can integrate by parts in the integral defining $a_\perp(\phi^0,q^\eps)$:
\begin{align*}
-a_\perp(\phi^0,q^\eps)
&=-\int_\Omega A_\perp\nabla_\perp\phi^0\cdot\nabla_\perp q^\eps dx 
\\
&=-\int_{\partial\Omega_N} (Id-bb^t)A_\perp\nabla_\perp\phi^0\cdot n q^\eps d\sigma
+\int_\Omega (\nabla_\perp\cdot A_\perp\nabla_\perp\phi^0) q^\eps dx
\\
&\le C||\phi^0||_{H^2} \left( ||q^\eps||_{L^2(\partial\Omega_N)}+||q^\eps||_{L^2(\Omega)} \right)
\end{align*}
since $\nabla\phi^0$ has a trace on $\partial\Omega$ and its norm in $L^2(\partial\Omega_N)$ is bounded by $C||\phi^0||_{H^2}$. Thus,
$$
\frac 1 \eps ||\nabla_{||}q^\eps||^2_{L^2} \le \frac C \eps a_{||}(q^\eps,q^\eps) 
\le C||f||_{L^2} ||q^\eps||_{L^2(\Omega)} + C||\phi^0||_{H^2} \left( ||q^\eps||_{L^2(\partial\Omega_N)}+||q^\eps||_{L^2(\Omega)} \right).
$$
By Poincar\'e-Wirtinger inequality (\ref{PoinW}) (note that $Pq^\eps=0$) and by Hypothesis B' we have 
$$
\max(||q^\eps||_{L^2(\Omega)},  ||q^\eps||_{L^2(\partial\Omega_N)}) \le C ||\nabla_{||}q^\eps||_{L^2}
$$
so that
$$
\frac 1 \eps ||\nabla_{||}q^\eps||_{L^2} \le C\max(||f||_{L^2},||\phi^0||_{H^2}).
$$
This is the same as (\ref{lamBound}) since $\nabla_{||}\lambda^\eps = \frac 1 \eps \nabla_{||}q^\eps$ according to (\ref{NR0}) and (\ref{NR}).
\end{proof}

\bigskip
\begin{rem}
  The Limit model (\ref{eq:Jv9a}), reformulated using the Lagrange
  multiplier technique, now reads: Find $(\phi^0,\;\lambda^0) \in
  \mathcal V\times \mathcal L$ such that
  \begin{gather}
    (L')\,\,\,
    \left\{
      \begin{array}{lr}
        \displaystyle
        \int_\Omega  A_{\perp} \nabla_{\perp} \phi^{0}
        \cdot 
        \nabla_{\perp} \psi \, dx
        +
        \int_\Omega A_{||} \nabla_{||} \psi \cdot \nabla_{||} \lambda^{0}  \, dx
        =
        \int_\Omega f \psi \, dx
        & \forall \psi \in \mathcal V
        \\[3mm]
        \displaystyle
        \int_\Omega A_{||} \nabla_{||} \phi^{0} \cdot \nabla_{||} \kappa  \, dx =0
        & \forall \kappa \in \mathcal L\,.
      \end{array}
    \right.
    \label{eq:Jx9a}
  \end{gather}
  Problem (\ref{eq:Jx9a}) is also well posed assuming Hypotheses A, B, B' and $\phi^0 \in H^2(\Omega)$. Indeed, the uniqueness of the solution 
  to (\ref{eq:Jx9a}) can be proved in exactly the same
  manner as in the proof of Proposition \ref{lem:AP-aquiv2} above. To prove the existence of a solution,
  it suffices to take the limit $\eps\to 0$ in the first two lines of (\ref{eq:Ju7a}). Indeed, we know by Theorem \ref{thm_EX} that
  $p^\eps\to\phi^0$, the solution to (\ref{eq:Jv9a}), and $q^\eps\to 0$ in $H^1(\Omega)$. Moreover, the family $\{\nabla_{||}\lambda^\eps\}$ is bounded in the norm of $L^2(\Omega)$ 
  by Proposition \ref{lamBoundLem}. We can take therefore a weakly convergence subsequence $\{\nabla_{||}\lambda^{\eps_n}\}$ and identify its limit with $\{\nabla_{||}\lambda^0\}$
  with some $\lambda^0\in{\cal L}$ (cf. Lemma \ref{lemlem}) to see that $(\phi^0,\lambda^0)\in  \mathcal V\times \mathcal L$ solves (\ref{eq:Jx9a}).
\end{rem}

\color{black}

%%%%%%%%%%%%%%%%%%%%%%%
\section{Numerical method}\label{sec:num_met}
%%%%%%%%%%%%%%%%%%%%%%%

This section concerns the discretization of the Asymptotic Preserving
formulation (\ref{eq:Ju7a}), based on a finite element method, and the
detailed study of the obtained numerical results. The numerical
analysis of the present scheme is investigated in a forthcoming work
\cite{AJC}, 
in particular we are interested in the convergence of the
scheme, independently of the parameter $\eps >0$.\\

Let us denote by ${\cal V}_{h} \subset \mathcal{V}$ and ${\cal L}_{h}
\subset \mathcal{L}$ the finite dimensional approximation spaces,
constructed by means of appropriate numerical discretizations (see Section
\ref{Discr} and Appendix {\bf B}). We are thus looking for a discrete solution
$(p^\varepsilon_h,\;\lambda^\varepsilon_h,\;q^\varepsilon_h,\;l^\varepsilon_h,\;\mu^\varepsilon_h)
\in \mathcal V_h\times \mathcal L_h\times \mathcal V_h\times \mathcal
V_h \times \mathcal L_h$ of the following system
\begin{gather}
  \left\{
    \begin{array}{ll}
      \displaystyle
      a_{\perp} (p^\varepsilon_h , \eta ) +
      a_{\perp} (q^\varepsilon_h , \eta ) +
      a_{||} (\eta , \lambda^{\varepsilon }_h )
      = \left(f,\eta  \right) \,, \quad \forall \eta \in {\cal V}_h\,,
      \\[3mm]
      \displaystyle
      a_{||} ( p^{\varepsilon }_h ,  \kappa) =0  \,, \quad \forall \kappa \in {\cal L}_h  \,, 
      \\[3mm]
      \displaystyle
      a_{||} (q^\varepsilon_h , \xi ) +
      \varepsilon a_{\perp} (q^\varepsilon_h , \xi ) +
      \varepsilon a_{\perp} (p^\varepsilon_h , \xi ) +
      \left( l^{\varepsilon }_h , \xi \right)
      = \varepsilon \left(f,\xi  \right) \,, \quad \forall \xi \in {\cal V}_h\,,
      \\[3mm]
      \displaystyle
      \left( q^{\varepsilon }_h , \chi \right) +
      a_{||} (\chi , \mu^{\varepsilon }_h ) =0 \,, \quad \forall \chi \in {\cal V}_h\,,
      \\[3mm]
      \displaystyle
      a_{||} (l^{\varepsilon }_h , \tau ) =0 \,, \quad \forall \tau \in {\cal L}_h\,.
    \end{array}
  \right.
  \label{eq:Jt8a}
\end{gather}

Our numerical experiments indicate that the spaces ${\cal V}_{h}$ and
$\mathcal{L}_h$ can be always taken of the same type and on the same
mesh.  The only difference between these two finite element spaces
lies thus in the incorporation of boundary conditions.  In general,
let $X_h$ denote the complete finite element space (without any
restrictions on the boundary) which should be $H^1$ conforming but
otherwise arbitrarily chosen.  We define then \be\label{VH}
\mathcal{V}_h=\{v_h\in X_h/v_h|_{\partial\Omega_{D}}=0\}, \ee
\be\label{LH} \mathcal{L}_h=\{\lambda_h\in
X_h/\lambda_h|_{\partial\Omega_{in} \cup \partial\Omega_{D}}=0\}\,.
\ee While this choice of $\mathcal{V}_h$ is straight forward, the
boundary conditions in $\mathcal{L}_h$ require special
attention. Indeed, nothing in the definition (\ref{Lsp}) of space
$\mathcal{L}$ on the continuous level indicates that its elements
should vanish on $\partial\Omega_{D}$. However, this liberty on
$\partial\Omega_{D}$ is somewhat counter-intuitive. Indeed, the
Lagrange multiplier $\lambda^\eps\in\mathcal{L}$ serves to impose
$\nabla_{||}p^\eps=0$ for some function $p^\eps$ taken from the space
$\mathcal{V}$. But, for $p\in\mathcal{V}$ the trace on
$\partial\Omega_{D}$ is zero so that $\nabla_{||}p^\eps=0$ there
without the help of a Lagrange multiplier.  Of course, this argument
is not valid on the continuous level since the trace of functions in
$\mathcal{L}$ does not even necessarily exist. However, this may
become very important on the finite element level. Indeed, we provide
in Appendix \ref{appB} an example of a finite element setting without
incorporating $\lambda_h|_{\partial\Omega_{D}}=0$ into the definition
of $\mathcal{L}_h$, which leads to an ill-posed system
(\ref{eq:Jt8a}). To avoid this difficulty, we choose $\mathcal{L}_h$
as in (\ref{LH}) in all our experiments, thus obtaining well-posed
problems.

%%%%%%%%%%%%%%%%%%%%%%%
\subsection{Discretization} \label{Discr}
%%%%%%%%%%%%%%%%%%%%%%%
Let us present the discretization in a 2D case, the 3D case being a
simple generalization. The here considered computational domain
$\Omega $ is a square $\Omega = [0,1]\times [0,1]$. All simulations
are performed on structured meshes. Let us introduce the Cartesian,
homogeneous grid
\begin{gather}
  x_i = i / N_x \;\; , \;\; 0 \leq i \leq N_x \,, \quad
  y_j = j / N_y \;\; , \;\; 0 \leq j \leq N_y
  \label{eq:Jp8a},
\end{gather}
where $N_x$ and $N_y$ are positive even constants, corresponding to
the number of discretization intervals in the $x$-
resp. $y$-direction. The corresponding mesh-sizes are denoted by $h_x
>0$ resp. $h_y >0$. Choosing a $\mathbb Q_2$ finite element method
($\mathbb Q_2$-FEM), based on the following quadratic base functions
%At every grid point we define a corresponding $\mathbb Q_2$ base
%function $\eta_{h,ij} = \eta _{xi} \eta _{yj} $ that takes value $1$
%at the grid point:

\begin{gather}
  \theta _{x_i}=
  \left\{
    \begin{array}{ll}
      \frac{(x-x_{i-2})(x-x_{i-1})}{2h_x^{2}} & x\in [x_{i-2},x_{i}],\\
      \frac{(x_{i+2}-x)(x_{i+1}-x)}{2h_x^{2}} & x\in [x_{i},x_{i+2}],\\
      0 & \text{else}
    \end{array}
  \right.\,, \quad 
  \theta _{y_j} =
  \left\{
    \begin{array}{ll}
      \frac{(y-y_{j-2})(y-y_{j-1})}{2h_y^{2}} & y\in [y_{j-2},y_{j}],\\
      \frac{(y_{j+2}-y)(y_{j+1}-y)}{2h_y^{2}} & y\in [y_{j},y_{j+2}],\\
      0 & \text{else}
    \end{array}
  \right.
  \label{eq:Js8a1}
\end{gather}
for even $i,j$  and
\begin{gather}
  \theta _{x_i}=
  \left\{
    \begin{array}{ll}
      \frac{(x_{i+1}-x)(x-x_{i-1})}{h_x^{2}} & x\in [x_{i-1},x_{i+1}],\\
      0 & \text{else}
    \end{array}
  \right.\,, \quad 
  \theta _{y_j} =
  \left\{
    \begin{array}{ll}
      \frac{(y_{j+1}-y)(y-y_{j-1})}{h_y^{2}} & y\in [y_{j-1},y_{j+1}],\\
      0 & \text{else}
    \end{array}
  \right.
  \label{eq:Js8a2}
\end{gather}
for odd $i,j$, we define
$$
X_h := \{ v_h = \sum_{i,j} v_{ij}\,  \theta_{x_i} (x)\,  \theta_{y_j}(y)\}\,,
$$
We then search for discrete solutions $(p^\varepsilon_h,\;q^\varepsilon_h,\;l^\varepsilon_h)
\in \mathcal V_h\times \mathcal V_h\times \mathcal
V_h$  and $(\lambda^\varepsilon_h,\;\mu^\varepsilon_h)
\in \mathcal L_h\times  \mathcal L_h$ with
${\cal V}_h $ and ${\cal L}_h $ defined by (\ref{VH}) and (\ref{LH}).
This leads  to the inversion of a linear system, the corresponding
matrix being non-symmetric and given by
\begin{gather}
  A = 
  \left(
    \begin{array}{ccccc}
      A_1    & A_0 & A_1 & 0 & 0 \\
      A_0 & 0 & 0 & 0 & 0 \\
      \varepsilon A_1 & 0 & A_0 +\varepsilon A_1& C & 0 \\
      0 & 0 & C & 0 & A_0\\
      0 & 0 & 0 & A_0& 0
    \end{array}
  \right)
  \label{eq:Jo8a}
  .
\end{gather}
The sub-matrices $A_0$, $A_1$ resp. $C$ correspond to the bilinear
forms $a_{||}(\cdot,\cdot)$, $a_{\perp}(\cdot,\cdot)$ resp. $(\cdot,\cdot)$,
used in equations (\ref{eq:Ju7a}) and belong to $\mathbb
R^{(N_x+1)(N_y+1)\times (N_x+1)(N_y+1)}$. The matrix elements are
computed using the 2D Gauss quadrature formula, with 3 points in the
$x$ and $y$ direction:
\begin{gather}
  \int_{-1}^{1}\int_{-1}^{1}f (x,y) =
  \sum_{i,j=-1}^{1} \omega _{i}\omega _{j} f (x_i,y_j)\,,
  \label{eq:Jk9a}
\end{gather}
where $x_0=y_0=0$, $x_{\pm 1}=y_{\pm 1}=\pm\sqrt {\frac{3}{5}}$,
$\omega _0 = 8/9$ and $\omega _{\pm 1} = 5/9$, which is exact for
polynomials of degree 5.

\subsection{Numerical Results}\label{sec:test case}
%\subsubsection{2D test case, $b = const.$}
\subsubsection{2D test case,  uniform and aligned $b$-field}
In this section we compare the numerical results obtained via the
$\mathbb Q_2$-FEM, by discretizing the Singular Perturbation
model (\ref{eq:J07a}), the Limit model (\ref{eq:Jv9a}) and the
Asymptotic Preserving reformulation (\ref{eq:Ju7a}). In all numerical
tests we set $A_\perp = Id$ and $A_\parallel = 1$. We start with a
simple test case, where the analytical solution is known. Let the
source term $f$ be given by
\begin{gather}
  f = \left(4 + \varepsilon  \right) \pi^{2} \cos \left( 2\pi x\right)
  \sin \left(\pi y \right) +
  \pi^{2} \sin \left(\pi y \right)
  \label{eq:J87a}
\end{gather}
and the $b$ field be aligned with the $x$-axis. Hence, the solution
$\phi^{\varepsilon }$ of (\ref{eq:J07a}) and its decomposition
$\phi^{\varepsilon }=p^{\varepsilon } + q^{\varepsilon }$ write
\begin{gather}
  \phi^{\varepsilon } = \sin \left(\pi y \right) + \varepsilon \cos \left( 2\pi x\right)
  \sin \left(\pi y \right), \\
  p^{\varepsilon } = \sin \left(\pi y \right)\,, \quad 
  q^{\varepsilon } = \varepsilon \cos \left( 2\pi x\right)
  \sin \left(\pi y \right)
  \label{eq:J97a}.
\end{gather}

We denote by $\phi _P$, $\phi _L$, $\phi _A$ the numerical solution of
the Singular Perturbation model (\ref{eq:J07a}), the Limit model
(\ref{eq:Jv9a}) and the Asymptotic Preserving reformulation
(\ref{eq:Ju7a}) respectively. The comparison will be done in the
$L^{2}$-norm as well as the  $H^{1}$-norm. The linear systems obtained after
discretization of the three methods are solved using the same
numerical algorithm --- LU decomposition implemented in a solver
MUMPS\cite{MUMPS}.

\def\xxxa{0.45\textwidth}
\begin{figure}[!ht] 
  \centering
  \subfigure[$L^{2}$ error for a grid with $50\times 50$ points.]
  {\includegraphics[angle=-90,width=\xxxa]{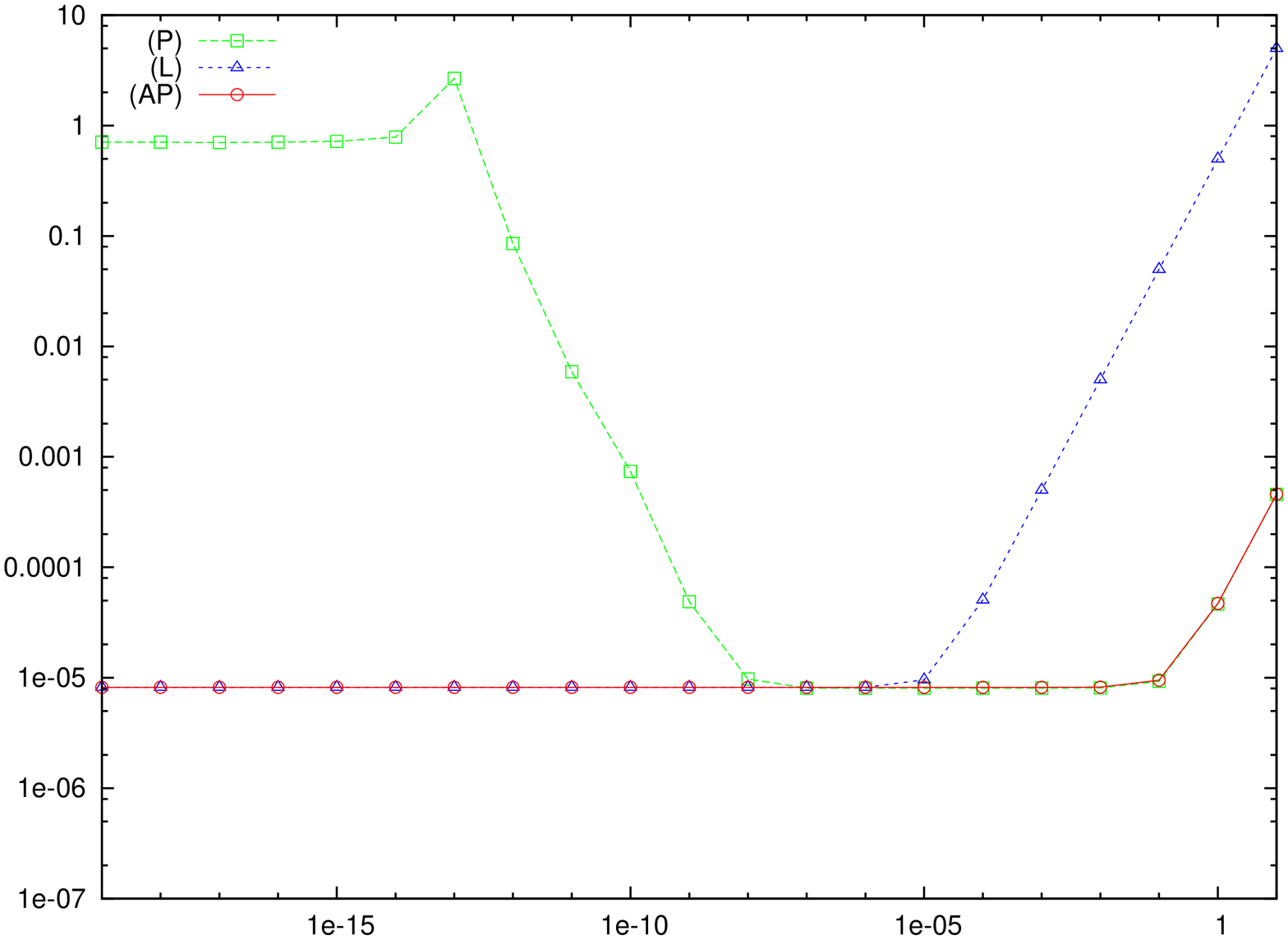}}
  \subfigure[$H^{1}$ error for a grid with $50\times 50$ points.]
  {\includegraphics[angle=-90,width=\xxxa]{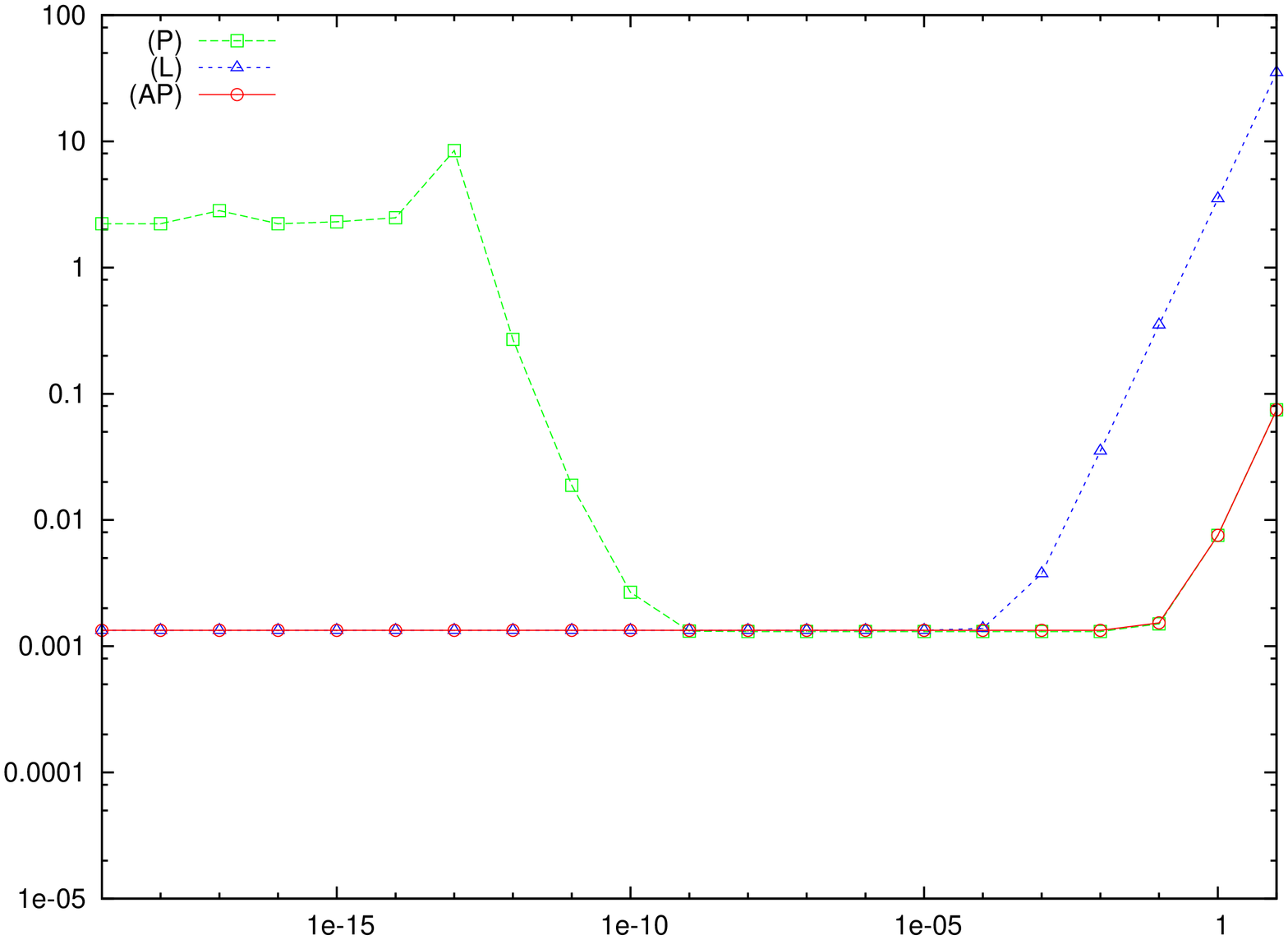}}

  \subfigure[$L^{2}$ error for a grid with $100\times 100$ points.]
  {\includegraphics[angle=-90,width=\xxxa]{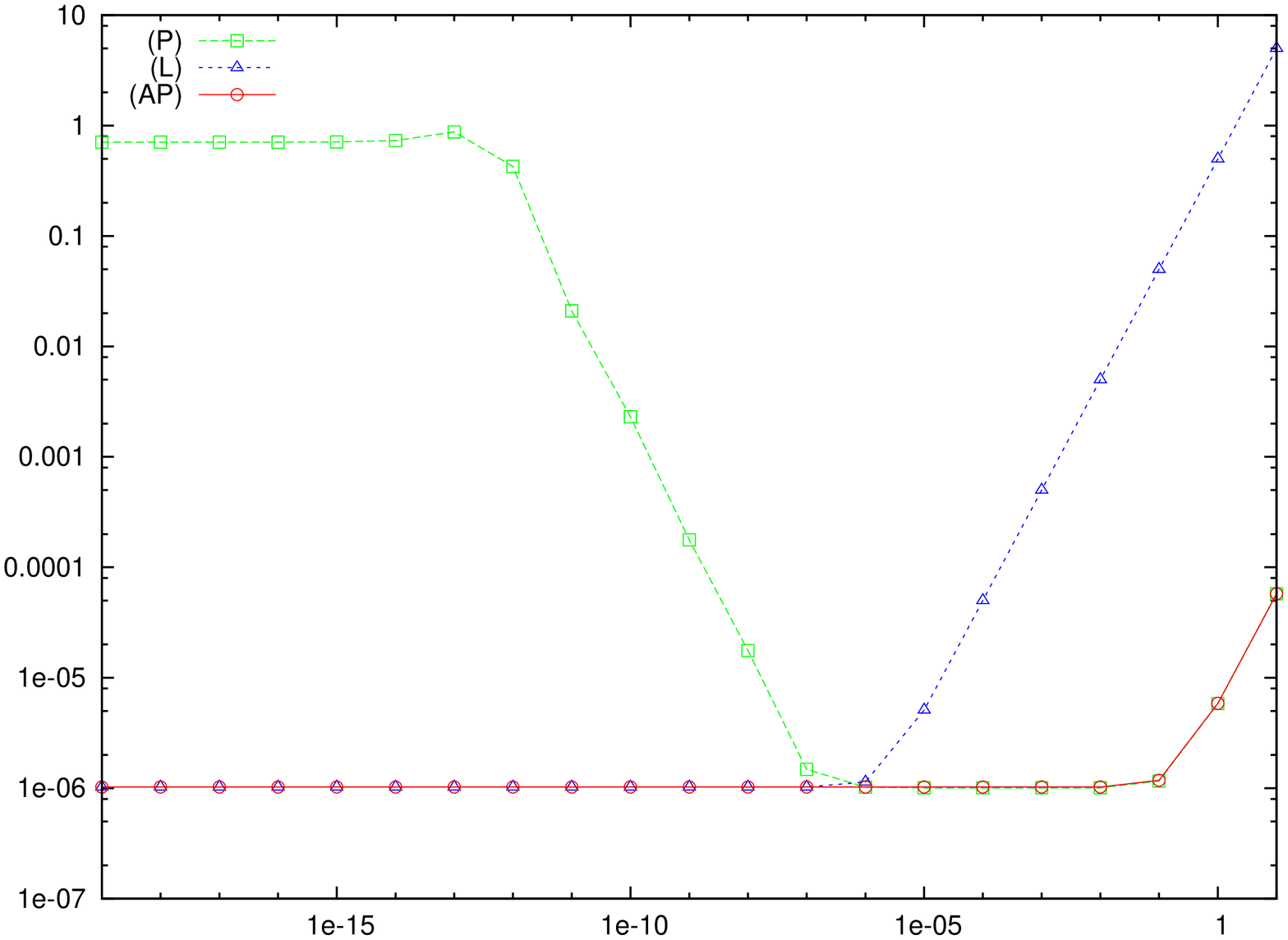}}
  \subfigure[$H^{1}$ error for a grid with $100\times 100$ points.]
  {\includegraphics[angle=-90,width=\xxxa]{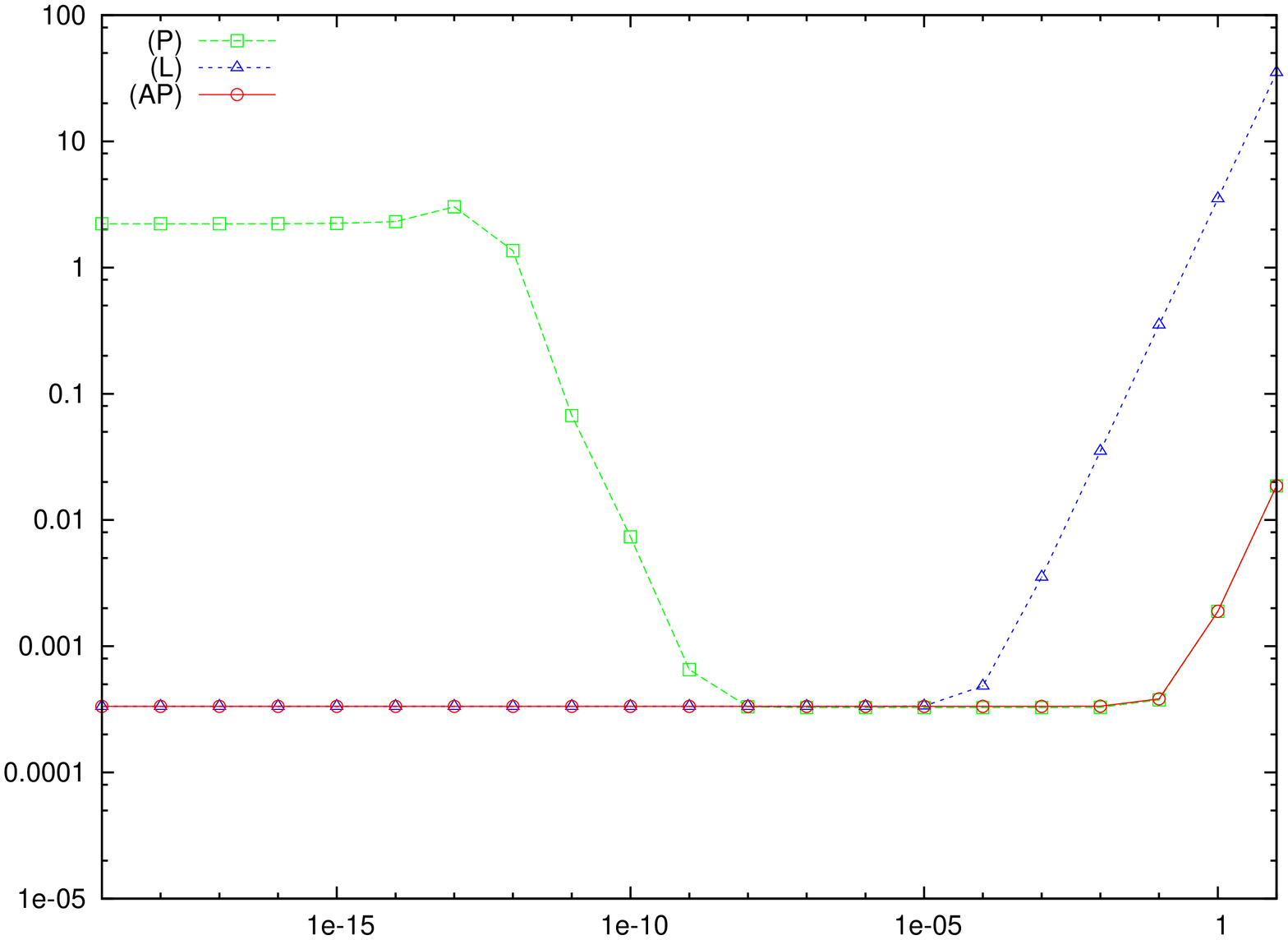}}

  \subfigure[$L^{2}$ error for a grid with $200\times 200$ points.]
  {\includegraphics[angle=-90,width=\xxxa]{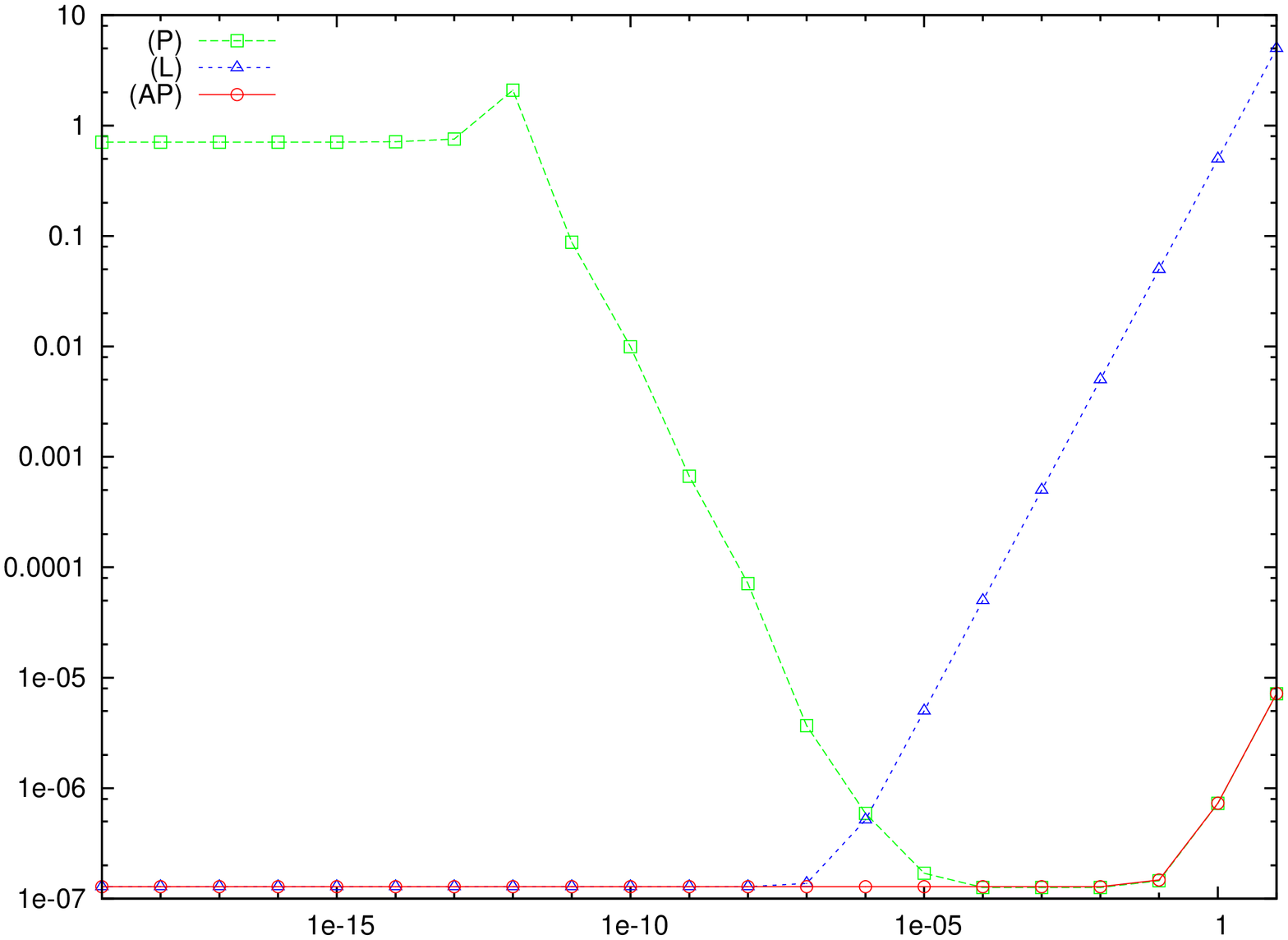}}
  \subfigure[$H^{1}$ error for a grid with $200\times 200$ points.]
  {\includegraphics[angle=-90,width=\xxxa]{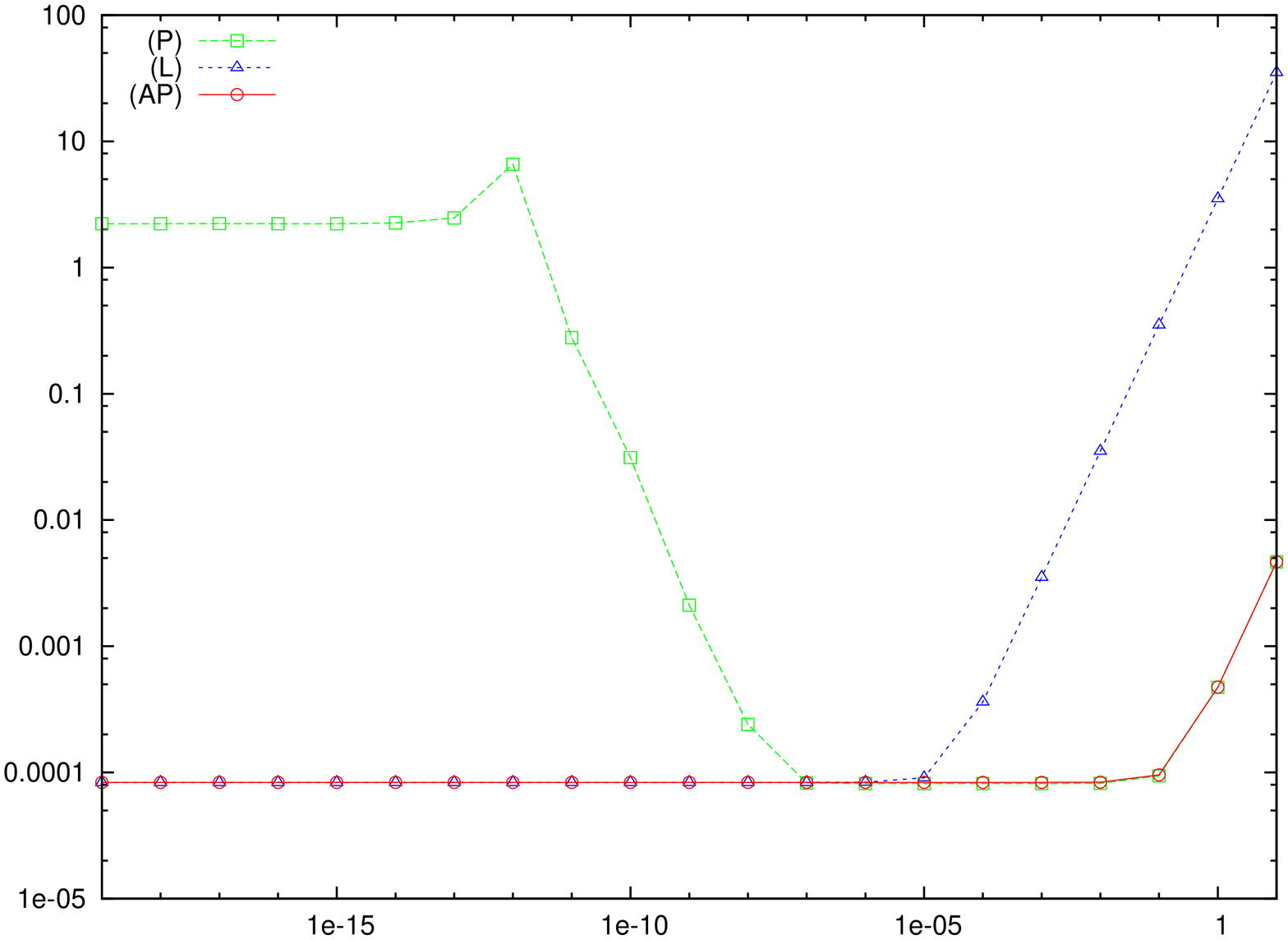}}

  \caption{Absolute $L^{2}$ (left column) and $H^{1}$ (right column)
    errors between the exact solution $\phi^{\varepsilon }$ and the
    computed numerical solution $\phi _A$ (AP), $\phi _L$ (L), $\phi
    _P$ (P) for the test case with constant $b$. The error is plotted
    as a function of the parameter $\eps$ and for three different
    mesh-sizes.}
  \label{fig:error}
\end{figure}

\begin{table}
  \centering
  \begin{tabular}{|c||c|c||c|c||c|c|}
    \hline
    \multirow{2}{*}{$\varepsilon$} &
    \multicolumn{2}{|c||}{\rule{0pt}{2.5ex}AP scheme} 
    & \multicolumn{2}{|c||}{Limit model} 
    & \multicolumn{2}{|c|}{Singular Perturbation scheme} \\
    \cline{2-7} 
    &\rule{0pt}{2.5ex}
    $L^2$ error & $H^1$ error & $L^2$ error & $H^1$ error & $L^2$ error & $H^1$ error  \\
    \hline\hline\rule{0pt}{2.5ex}
    10  & 
    $7.2\times 10^{-6}$ & $4.7\times 10^{-3}$ &
    $5.0\times 10^{0}$ & $3.51\times 10^{1}$ &
    $7.2\times 10^{-6}$ & $4.7\times 10^{-3}$ 
    \\
    \hline\rule{0pt}{2.5ex} 
    1 & 
    $7.3\times 10^{-7}$ & $4.7\times 10^{-4}$ &
    $5.0\times 10^{-1}$ & $3.51\times 10^{0}$ &
    $7.3\times 10^{-7}$ & $4.7\times 10^{-4}$ 
    \\
    \hline\rule{0pt}{2.5ex} 
    $10^{-1}$&
    $1.47\times 10^{-7}$ & $9.6\times 10^{-5}$ &
    $5.0\times 10^{-2}$ & $3.51\times 10^{-1}$ &
    $1.45\times 10^{-7}$ & $9.4\times 10^{-5}$ 
    \\
    \hline\rule{0pt}{2.5ex} 
    $10^{-4}$&
    $1.28\times 10^{-7}$ & $8.3\times 10^{-5}$ &
    $5.0\times 10^{-5}$ & $3.61\times 10^{-4}$ &
    $1.26\times 10^{-7}$ & $8.2\times 10^{-5}$ 
    \\
    \hline\rule{0pt}{2.5ex} 
    $10^{-6}$&
    $1.28\times 10^{-7}$ & $8.3\times 10^{-5}$ &
    $5.2\times 10^{-7}$ & $8.4\times 10^{-5}$ &
    $5.9\times 10^{-7}$ & $8.2\times 10^{-5}$ 
    \\
    \hline\rule{0pt}{2.5ex} 
    $10^{-10}$&
    $1.28\times 10^{-7}$ & $8.3\times 10^{-5}$ &
    $1.28\times 10^{-7}$ & $8.3\times 10^{-5}$ &
    $9.9\times 10^{-3}$ & $3.12\times 10^{-2}$ 
    \\
    \hline\rule{0pt}{2.5ex} 
    $10^{-15}$&
    $1.28\times 10^{-7}$ & $8.3\times 10^{-5}$ &
    $1.28\times 10^{-7}$ & $8.3\times 10^{-5}$ &
    $7.1\times 10^{-1}$ & $2.23\times 10^{0}$ 
    \\
    \hline
  \end{tabular}
  \caption{Comparison between the Asymptotic Preserving scheme, the
    Limit model and the Singular Perturbation model for $h=0.005$ (200 mesh points
    in each direction) and constant $b$: absolute $L^2$-error and $H^1$-error, for different $\eps$-values.}
  \label{tab:error}
\end{table}

In Figure \ref{fig:error} we plotted the absolute errors (in the $L^2$
resp. $H^1$-norms) between the numerical solutions obtained with one
of the three methods and the exact solution, and this, as a function
of the parameter $\eps$ and for several mesh-sizes. In Table
\ref{tab:error}, we specified the error values for one fixed grid and
several $\eps$-values. One observes that the Singular Perturbation
finite element approximation is accurate only for $\varepsilon$ bigger
than some critical value $\varepsilon _P$, the Limit model gives
reliable results for $\varepsilon$ smaller than $\varepsilon _L$,
whereas the AP-scheme is accurate independently on $\eps$.  The order
of convergence for all three methods is three in the $L^2$-norm and
two in the $H^1$-norm, which is an optimal result for $\mathbb Q_2$
finite elements.  When designing a robust numerical method one has
therefore two options. The first one is to use an Asymptotic
Preserving scheme, which is accurate independently on $\varepsilon $,
but requires the solution of a bigger linear system.  The second one
is to design a coupling strategy that involves the solution of the
Singular Perturbation formulation and the Limit problem in their
respective validity domains. This is however a very delicate problem,
since we observe that the critical values $\varepsilon _P$ and
$\varepsilon _L$ are mesh dependent, namely $\varepsilon _P$ inversely
proportional to $h$ and $\varepsilon _L$ proportional to
$h^{}$. Therefore for small meshes there may exist a range of
$\varepsilon $-values, where neither the Singular Perturbation nor the
Limit model finite element approximation give accurate results. For
our test case, this is even the case for meshes as big as $200\times
200$ points, if one regards the $L^2$-norm. This mesh-size is
generally insufficient in the case of real physical applications.

\begin{table}
  \centering
  \begin{tabular}{|c||c|c|c|c|c|}
    \hline\rule{0pt}{2.5ex}
    method & \# rows & \# non zero & time & $L^{2}$-error & $H^{1}$-error\\
    \hline
    \hline\rule{0pt}{2.5ex}
    AP &
    $50\times 10^{3}$ &
    $1563\times 10^{3}$ &
    $13.212$ s &
    $1.02\times 10^{-6}$ &
    $3.34\times 10^{-4}$
    \\
    \hline\rule{0pt}{2.5ex}
    L &
    $20\times 10^{3}$ &
    $469\times 10^{3}$ &
    $5.227$ s &
    $1.14\times 10^{-6}$ &
    $3.34\times 10^{-4}$
    \\
    \hline\rule{0pt}{2.5ex}
    P &
    $10\times 10^{3}$ &
    $157\times 10^{3}$ &
    $3.707$ s &
    $1.02\times 10^{-6}$ &
    $3.27\times 10^{-4}$
    \\
    \hline
  \end{tabular}
  \caption{Comparison between the Asymptotic Preserving scheme (AP),
    the Limit model (L) and the Singular Perturbation model (P) for
    $h=0.01$ (100 mesh points in each direction) and fixed $\varepsilon =
    10^{-6}$: matrix size, number of nonzero elements, average
    computational time and error in $L^{2}$ and $H^{1}$ norms.}
  \label{tab:time}
\end{table}

Another interesting aspect with respect to which the three methods
must be compared, is the computational time and the size of the
matrices involved in the linear systems. Table \ref{tab:time} shows
that the Asymptotic Preserving scheme is expensive in computational
time and memory requirements, as compared to the other
methods. Indeed, the computational time required to solve the problem
is almost four times bigger than that of the Singular Perturbation
scheme. Moreover, the Asymptotic Preserving method involves matrices
that have five times more rows and ten times more nonzero elements
than the matrices obtained with the Singular Perturbation
approximation. It is however the only scheme that provides the
$h$-convergence regardless of $\varepsilon$. In order to reduce the
computational costs, a coupling strategy for problems with variable
$\varepsilon $ will be proposed in a forthcoming paper. In sub-domains
where $\varepsilon > \varepsilon _P$ the Singular Perturbation problem
will be solved, in sub-domains where $\varepsilon < \varepsilon _L$
the Limit problem will be solved and only in the remaining part, where
neither the Limit nor the Singular Perturbation model are valid, the
Asymptotic Preserving formulation will be solved.

%\subsubsection{2D test case, $b \neq const.$}
\subsubsection{2D test case, non-uniform and non-aligned $b$-field}
We now focus our attention on the original feature of the here
introduced numerical method, namely its ability to treat nonuniform
$b$ fields. In this section we present numerical simulations performed
for a variable field $b$.

First, let us construct a numerical test case. Finding an analytical
solution for an arbitrary $b$ presents a considerable difficulty. We
have therefore chosen a different approach. First, we choose a limit
solution
\begin{gather}
  \phi^{0} = \sin \left(\pi y +\alpha (y^2-y)\cos (\pi x) \right)
  \label{eq:J79a},
\end{gather}
where $\alpha $ is a numerical constant aimed at controlling the
variations of $b$. For $\alpha =0$, the limit solution of the previous
section is obtained. The limit solution for $\alpha =2$ is shown in
Figure \ref{fig:limit}. \textcolor{black}{We set $\alpha =2$ in what follows.}
\begin{figure}[!ht] 
  \centering
  \def\xxxa{0.45\textwidth}
  \includegraphics[width=\xxxa]{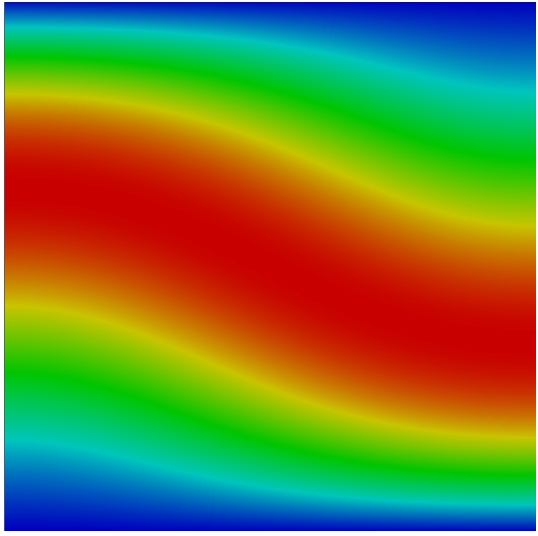}
  \caption{The limit solution for the test case with variable $b$.}
  \label{fig:limit}
\end{figure}
Since $\phi^{0}$ is a limit solution, it is constant along the $b$
field lines. Therefore we can determine the $b$ field using the
following implication
\begin{gather}
  \nabla_{\parallel} \phi^{0} = 0 \quad \Rightarrow \quad
  b_x \frac{\partial \phi^{0}}{\partial x} +
  b_y \frac{\partial \phi^{0}}{\partial y} = 0\,,
  \label{eq:J89a}
\end{gather}
which yields for example
\begin{gather}
  b = \frac{B}{|B|}\, , \quad
  B =
  \left(
    \begin{array}{c}
      \alpha  (2y-1) \cos (\pi x) + \pi \\
      \pi \alpha  (y^2-y) \sin (\pi x)
    \end{array}
  \right)
  \label{eq:J99a}\,\quad. 
\end{gather}
Note that the field $B$, constructed in this way, satisfies
$\text{div} B = 0$, which is an important property in the framework of
plasma simulation. Furthermore, we have $B \neq 0$ in the
computational domain. Now, we choose $\phi^\varepsilon $ to be a
function that converges, as $\varepsilon \rightarrow 0$, to the limit
solution $\phi^{0}$:
\begin{gather}
  \phi^{\varepsilon } = \sin \left(\pi y +\alpha (y^2-y)\cos (\pi
    x) \right) + \varepsilon \cos \left( 2\pi x\right) \sin \left(\pi
    y \right)
  \label{eq:Jc0a}.
\end{gather}
Finally, the force term is calculated, using the equation, i.e.
\begin{gather}
  f = - \nabla_\perp \cdot (A_\perp \nabla_\perp \phi^{\varepsilon })
  - \frac{1}{\varepsilon }\nabla_\parallel \cdot (A_\parallel \nabla_\parallel \phi^{\varepsilon })
  \nonumber.
\end{gather}

As in the previous section, we shall compare here the numerical
solution of the Singular Perturbation model (\ref{eq:J07a}), the Limit
model (\ref{eq:Jv9a}) and the Asymptotic Preserving reformulation
(\ref{eq:Ju7a}), i.e. $\phi _P$, $\phi _L$, $\phi _A$ with the exact
solution (\ref{eq:Jc0a}) . The $L^{2}$ and $H^{1}$-errors are reported
on Figure \ref{fig:errorvar} and Table \ref{tab:errorvar}. Once again
the Asymptotic Preserving scheme proves to be valid for all values of
$\varepsilon $, contrary to the other schemes. There is however a
difference compared to the constant-$b$ case.
%in the case of the Standard scheme --- the value of
%$\varepsilon _P$ changes. 
For a variable $b$ , the threshold value $\varepsilon _P$ seems to be
independent on the mesh size and is much larger than that of the
uniform $b$ test case. This observation limits further the possible
choice of coupling strategies, since even for coarse meshes there
exists a range of $\varepsilon $-values, where neither the Singular
Perturbation nor the Limit model are valid. The coupling strategy,
involving all three models, remains however interesting to
investigate.

\begin{table}
  \centering
  \begin{tabular}{|c||c|c||c|c||c|c|}
    \hline
    \multirow{2}{*}{$\varepsilon$} &
    \multicolumn{2}{|c||}{\rule{0pt}{2.5ex}AP scheme} 
    & \multicolumn{2}{|c||}{Limit model} 
    & \multicolumn{2}{|c|}{Singular Perturbation scheme} \\
    \cline{2-7} 
    &\rule{0pt}{2.5ex}
    $L^2$ error & $H^1$ error & $L^2$ error & $H^1$ error & $L^2$ error & $H^1$ error  \\
    \hline\hline\rule{0pt}{2.5ex}
    10  & 
    $7.2\times 10^{-6}$ & $4.6\times 10^{-3}$ &
    $5.0\times 10^{0}$ & $3.50\times 10^{1}$ &
    $7.2\times 10^{-6}$ & $4.6\times 10^{-3}$ 
    \\
    \hline\rule{0pt}{2.5ex} 
    1 & 
    $7.1\times 10^{-7}$ & $4.6\times 10^{-4}$ &
    $5.0\times 10^{-1}$ & $3.50\times 10^{0}$ &
    $7.1\times 10^{-7}$ & $4.6\times 10^{-4}$ 
    \\
    \hline\rule{0pt}{2.5ex} 
    $10^{-2}$&
    $2.05\times 10^{-7}$ & $1.33\times 10^{-4}$ &
    $5.0 \times 10^{-3}$ & $3.50\times 10^{-2}$ &
    $2.05\times 10^{-7}$ & $1.33\times 10^{-4}$ 
    \\
    \hline\rule{0pt}{2.5ex} 
    $10^{-4}$&
    $2.12\times 10^{-7}$ & $1.38\times 10^{-4}$ &
    $5.0 \times 10^{-5}$ & $3.77\times 10^{-4}$ &
    $1.74\times 10^{-6}$ & $1.43\times 10^{-4}$ 
    \\
    \hline\rule{0pt}{2.5ex} 
    $10^{-7}$&
    $2.17\times 10^{-7}$ & $1.41\times 10^{-4}$ &
    $2.22\times 10^{-7}$ & $1.41\times 10^{-4}$ &
    $1.68\times 10^{-3}$ & $1.26\times 10^{-2}$ 
    \\
    \hline\rule{0pt}{2.5ex} 
    $10^{-10}$&
    $2.17\times 10^{-7}$ & $1.41\times 10^{-4}$ &
    $2.17\times 10^{-7}$ & $1.41\times 10^{-4}$ &
    $3.93\times 10^{-1}$ & $1.35\times 10^{0}$ 
    \\
    \hline\rule{0pt}{2.5ex} 
    $10^{-15}$&
    $2.17\times 10^{-7}$ & $1.41\times 10^{-4}$ &
    $2.17\times 10^{-7}$ & $1.41\times 10^{-4}$ &
    $6.7 \times 10^{-1}$ & $2.32\times 10^{0}$ 
    \\
    \hline
  \end{tabular}
  \caption{Comparison between the Asymptotic preserving scheme, the
    Limit model and the Singular Perturbation model for $h=0.005$ (200 mesh points
    in each direction) and variable $b$: absolute $L^2$-error and $H^1$-error.}
  \label{tab:errorvar}
\end{table}

\def\xxxa{0.45\textwidth}
\begin{figure}[!ht] 
  \centering
  \subfigure[$L^{2}$ error for a grid with $50\times 50$ points.]
  {\includegraphics[angle=-90,width=\xxxa]{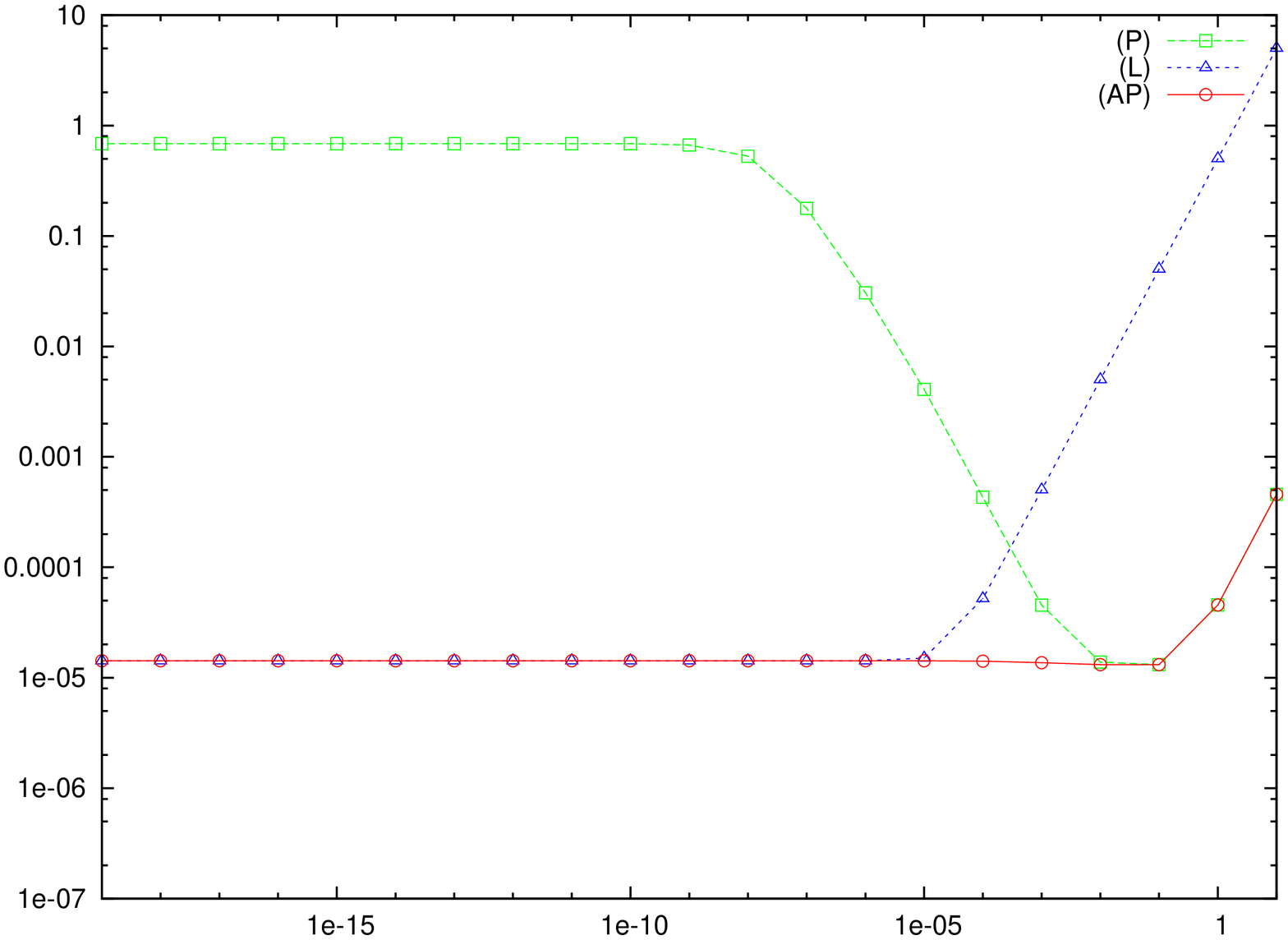}}
  \subfigure[$H^{1}$ error for a grid with $50\times 50$ points.]
  {\includegraphics[angle=-90,width=\xxxa]{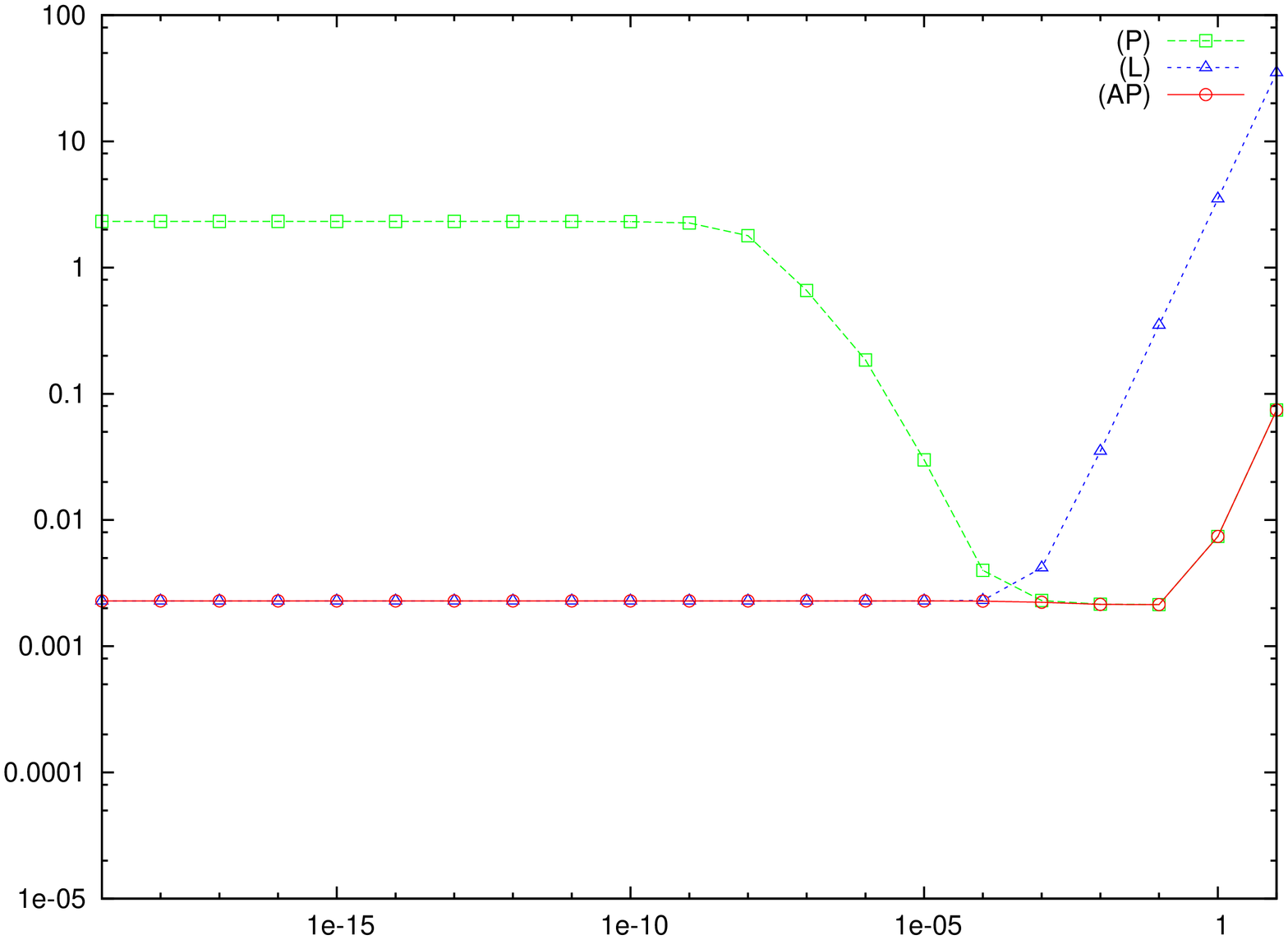}}

  \subfigure[$L^{2}$ error for a grid with $100\times 100$ points.]
  {\includegraphics[angle=-90,width=\xxxa]{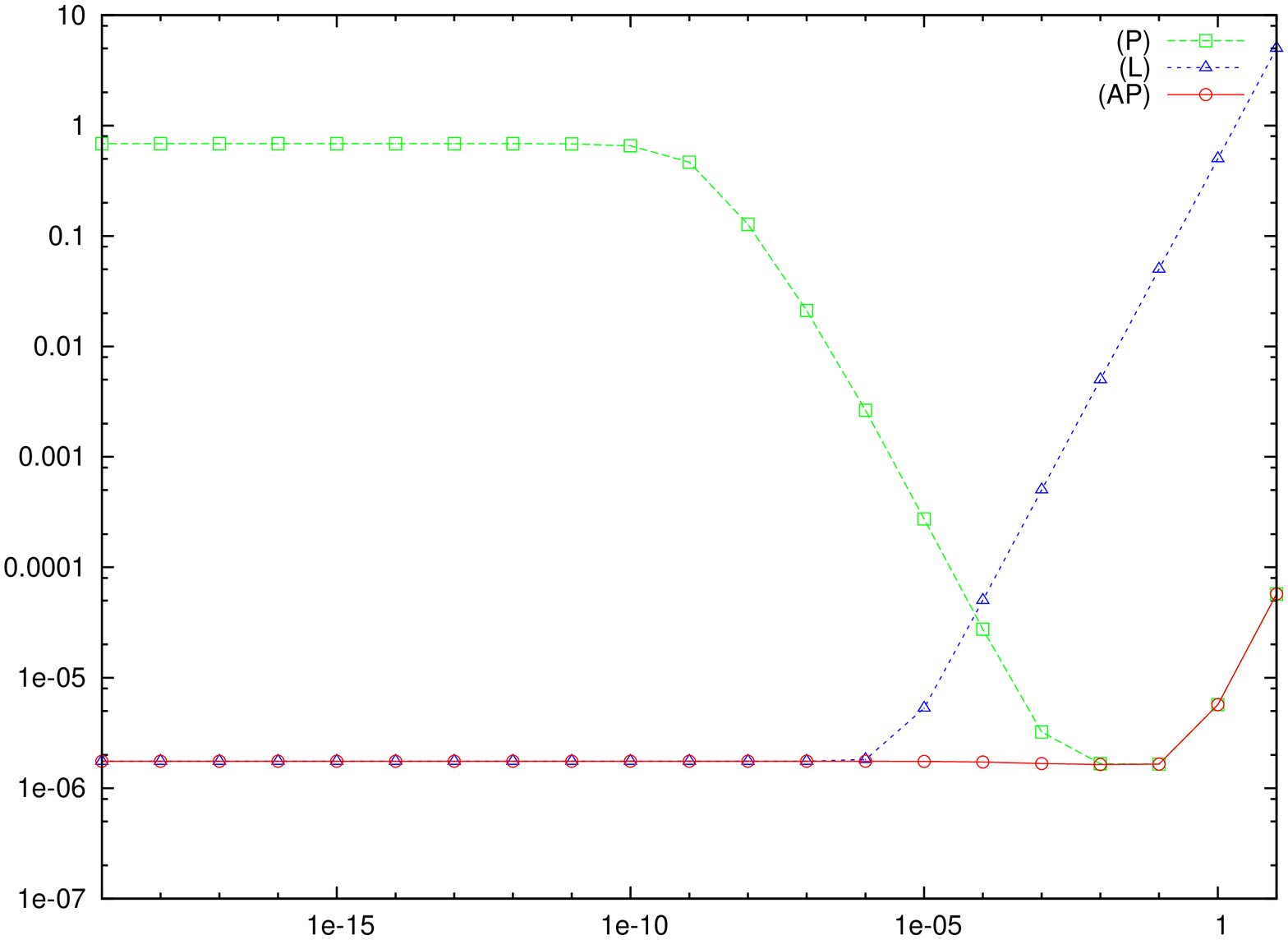}}
  \subfigure[$H^{1}$ error for a grid with $100\times 100$ points.]
  {\includegraphics[angle=-90,width=\xxxa]{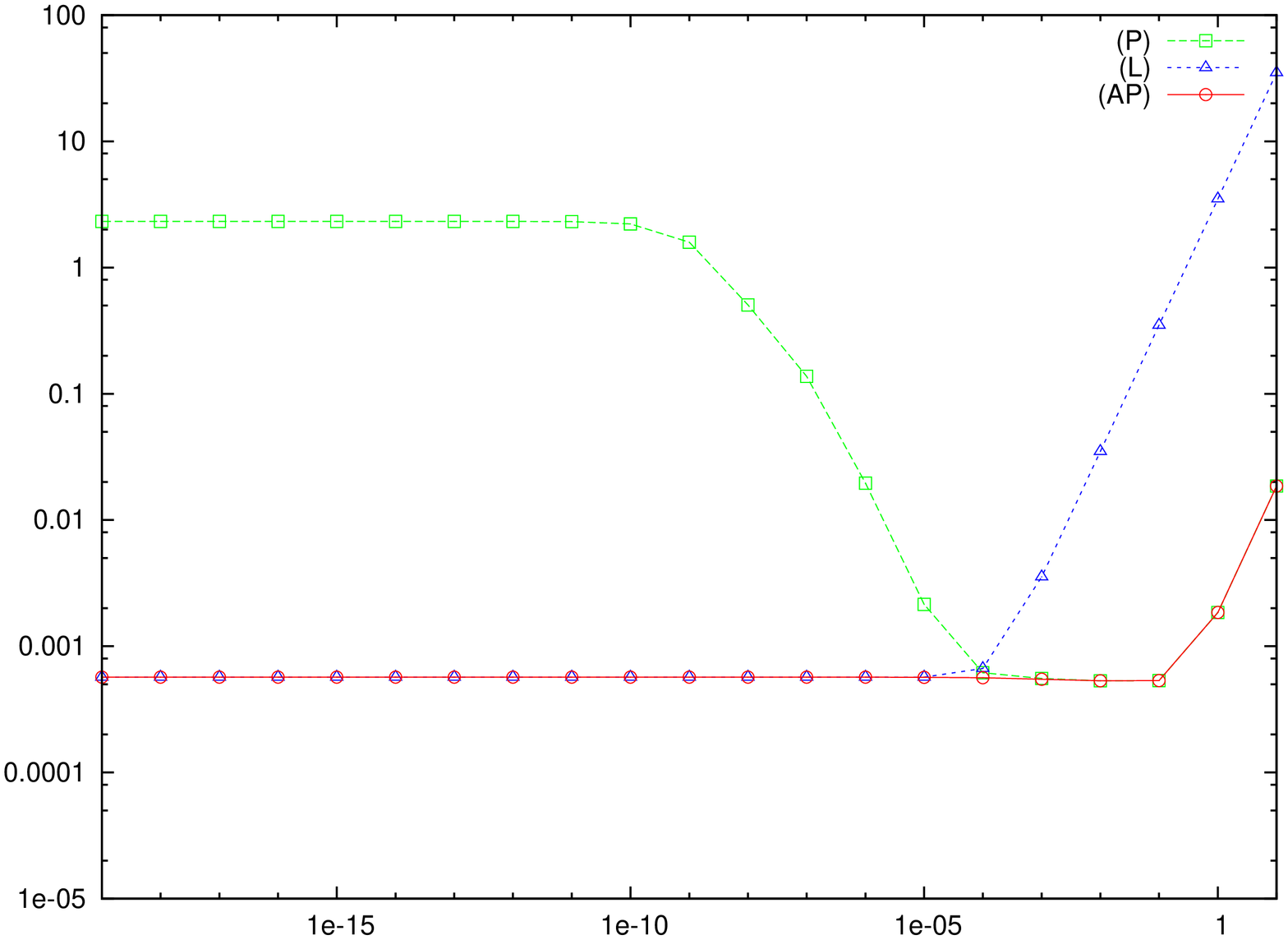}}

  \subfigure[$L^{2}$ error for a grid with $200\times 200$ points.]
  {\includegraphics[angle=-90,width=\xxxa]{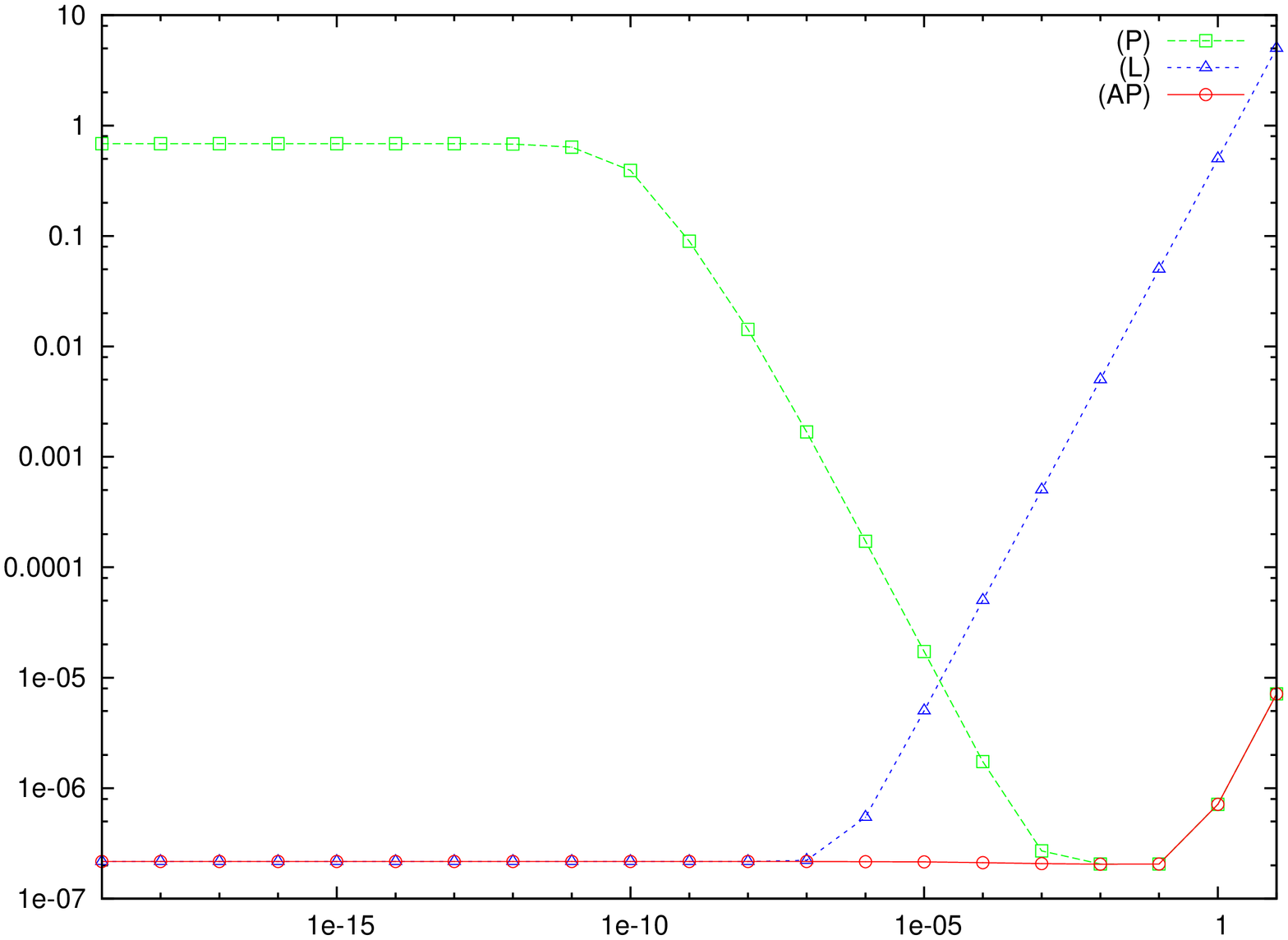}}
  \subfigure[$H^{1}$ error for a grid with $200\times 200$ points.]
  {\includegraphics[angle=-90,width=\xxxa]{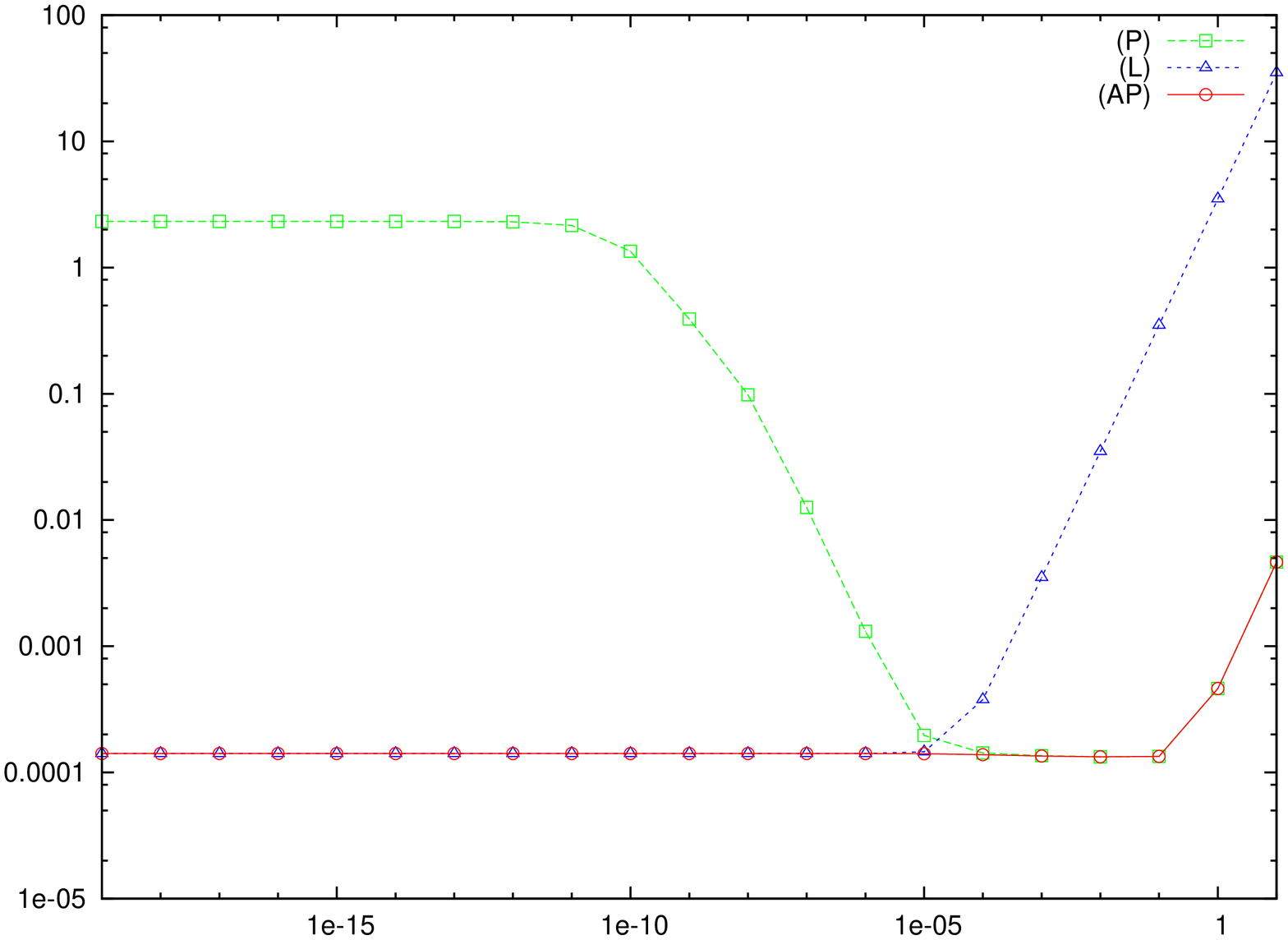}}

  \caption{Absolute $L^{2}$ (left column) and $H^{1}$ (right column)
    errors between the exact solution $\phi^{\varepsilon }$ and the
    computed solution $\phi _A$ (AP), $\phi _L$ (L), $\phi _P$ (P) for
    the test case with variable $b$. Plotted are the errors as a
    function of the small parameter $\varepsilon $, for three different
    meshes.}
  \label{fig:errorvar}
\end{figure}

In the next test case we investigate the influence of the variations
of the $b$ field on the accuracy of the solution. We would like to answer
the following question: what is the minimal number of points per
characteristic length of $b$ variations required to obtain an
acceptable solution. For this, let us modify the previous test case. Let $b =
B/|B|$, with
\begin{gather}
  B =
  \left(
    \begin{array}{c}
      \alpha  (2y-1) \cos (m\pi x) + \pi \\
      m\pi \alpha  (y^2-y) \sin (m\pi x)
    \end{array}
  \right)\, ,
  \label{eq:Ji0a}
\end{gather}
$m$ being an integer. The limit solution and $\phi^{\varepsilon}$ are
chosen to be
\begin{gather}
  \phi^{0} = \sin \left(\pi y +\alpha (y^2-y)\cos (m\pi x) \right)
  \label{eq:Jj0a}, \\
  \phi^{\varepsilon } = \sin \left(\pi y +\alpha (y^2-y)\cos (m\pi
    x) \right) + \varepsilon \cos \left( 2\pi x\right) \sin \left(\pi
    y \right)
  \label{eq:Jk0a}.
\end{gather}

We perform two tests: first, we fix the mesh size and vary $m$ to
find the minimal period of $b$ for which the Asymptotic Preserving
method yields still acceptable results. We define a result to be acceptable
when the relative error is less then $0.01$. In the second test $m$
remains fixed and the convergence of the scheme is studied. The
results are presented on Figures \ref{fig:bp_var} and
\ref{fig:bp_h_var}.

For $\varepsilon =1$ and 400 mesh points in each direction
($h=0.0025$) the relative error in the $L^{2}$-norm, defined as
$\frac{||\phi^{\varepsilon} -
  \phi_A||_{L^{2}(\Omega)}}{||\phi_A||_{L^{2}(\Omega)}}$, is below
$0.01$ for all tested values of $1 \leq m \leq 50$. The relative
$H^1$-error $\frac{||\phi^{\varepsilon} -
  \phi_A||_{H^{1}(\Omega)}}{||\phi_A||_{H^{1}(\Omega)}}$ exceeds the
critical value for $m > 25$. For $\varepsilon = 10^{-20}$ the maximal
$m$ for which the error is acceptable in both norms is $20$. The
minimal number of mesh points per period of $b$ variations is 40 in
the worst case, in order to obtain an $1\%$ relative error.

Figure \ref{fig:bp_h_var} and Table \ref{tab:bp_h_var} show the
convergence of the Asymptotic Preserving scheme with respect to $h$
for $m=10$ and $\varepsilon = 10^{-10}$. We observe that for big
values of $h$ the error does not diminish with $h$. Then, for $h <
0.025$ the scheme converges at a better rate then 2 for $H^{1}$-error
and 3 for $L^{2}$-error. For $h<0.00625$ (160 points) the optimal
convergence rate in the $H^{1}$-norm is obtained (which is 32 mesh
points per period of $b$). The method is super-convergent in the whole
tested range for the $L^{2}$-error.

These results are reassuring, as they prove that the Asymptotic
Preserving scheme is precise even for strongly varying fields for
relatively small mesh sizes, which was not evident. Indeed, the
optimal convergence rate in the $H^{1}$-norm is obtained for 32 mesh
points per $b$ period, and an $1\%$ relative error for 40 points. It
shows that accurate results can be obtained in more complex
simulations, such as tokamak plasma for example. The application of
the method to bigger scale problems is the subject of an ongoing work.

\def\xxxa{0.45\textwidth}
\begin{figure}[!ht]
  \centering
  \subfigure[$\varepsilon = 1$]
  {\includegraphics[angle=-90,width=\xxxa]{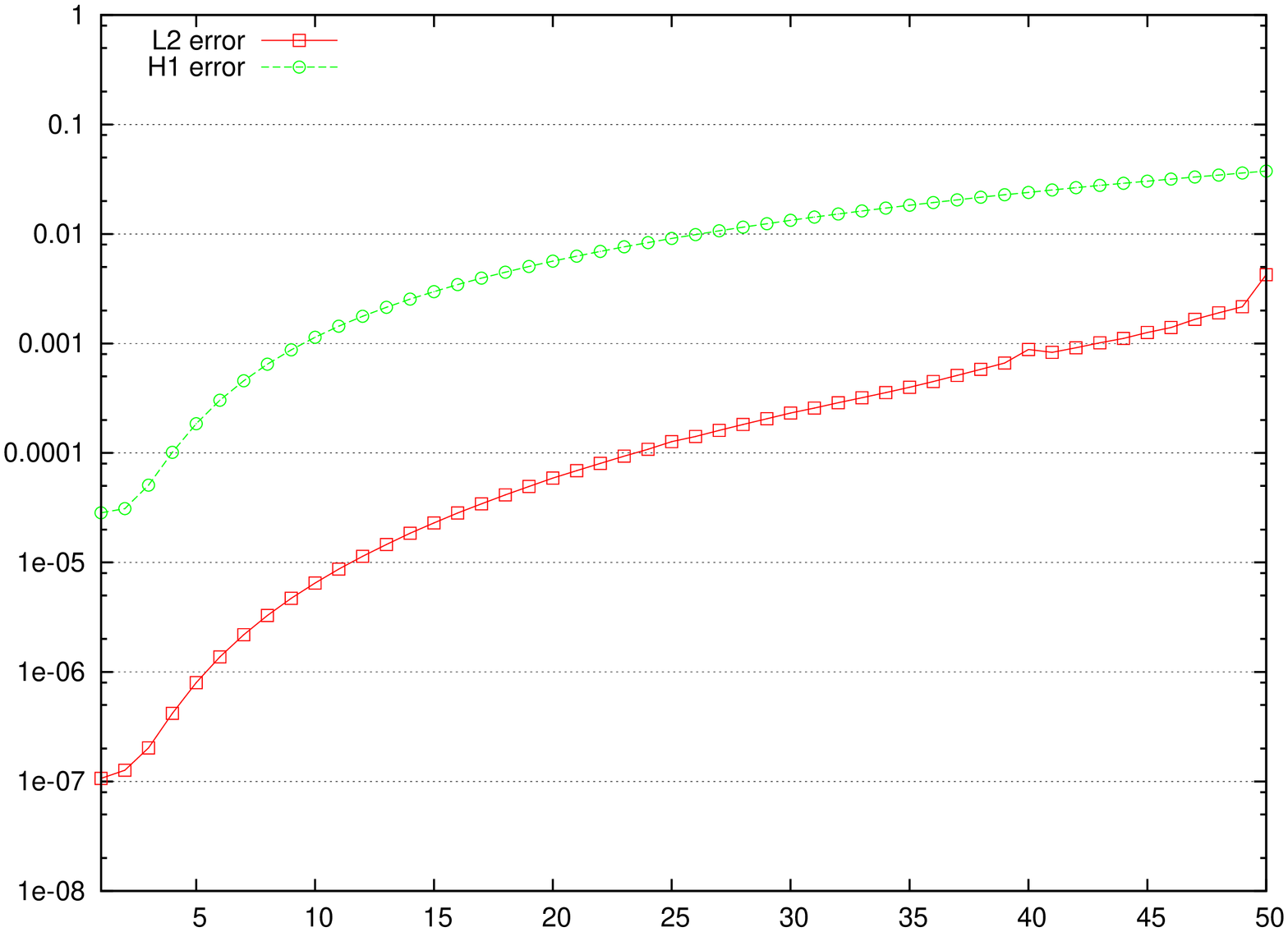}}
  \subfigure[$\varepsilon = 10^{-20}$]
  {\includegraphics[angle=-90,width=\xxxa]{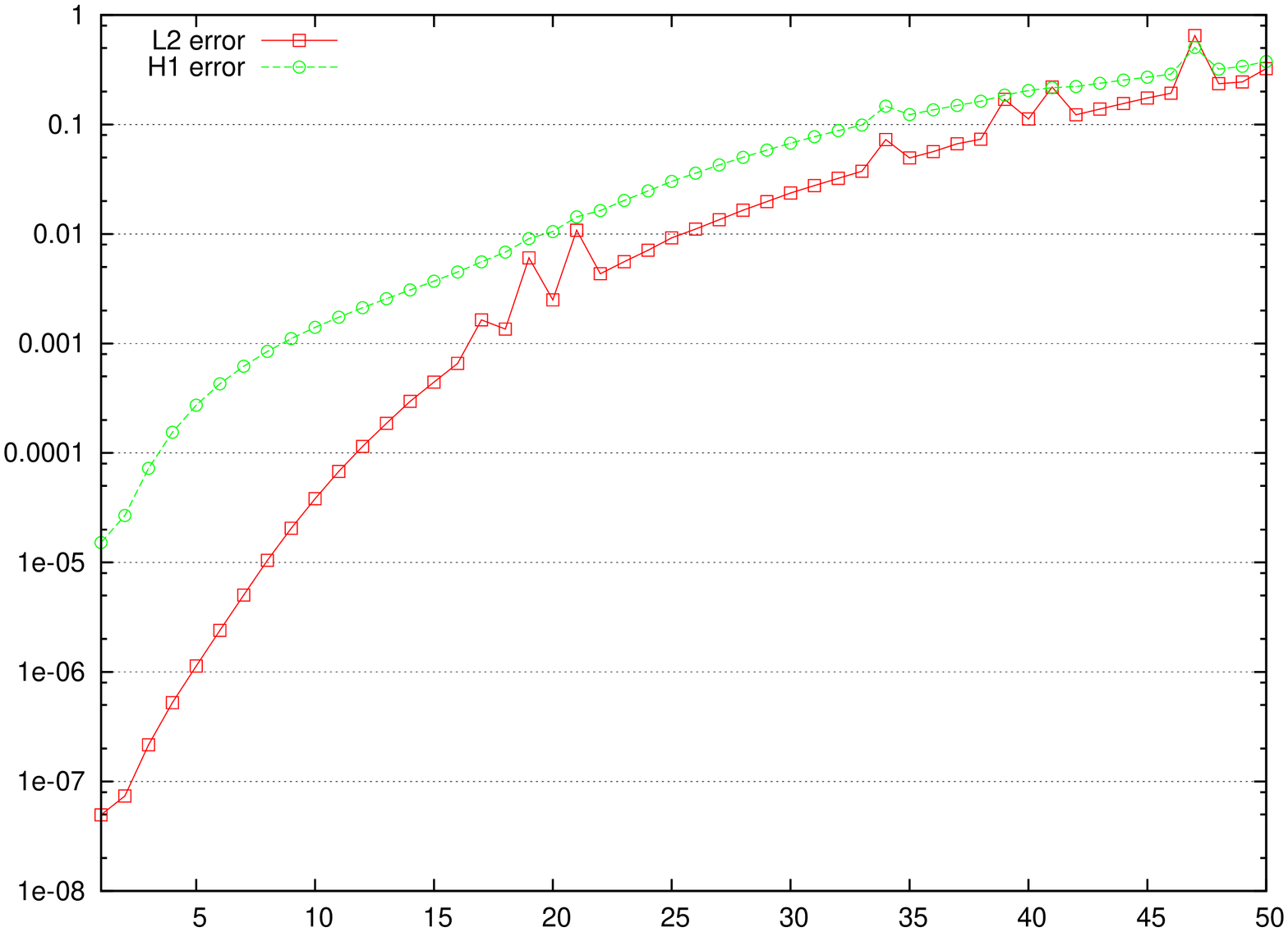}}

  \caption{Relative $L^{2}$ and $H^{1}$ errors between the exact
    solution $\phi^{\varepsilon }$ and the computed solution $\phi _A$
    (AP) for $h=0.0025$ (400 points in each direction) as a function of
    $m$ and for $\varepsilon = 1$ respectively  $10^{-20}$.}
  \label{fig:bp_var}
\end{figure}

\def\xxxa{0.45\textwidth}
\begin{figure}[!ht] 
  \centering
  {\includegraphics[angle=-90,width=\xxxa]{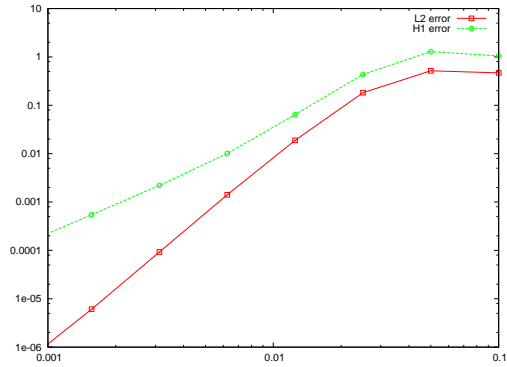}}
  
  \caption{Relative $L^{2}$ and $H^{1}$ errors between the exact
    solution $\phi^{\varepsilon }$ and the computed solution $\phi _A$
    (AP) for $m=10$ and $\varepsilon = 10^{-10}$ as a function of $h$.}
  \label{fig:bp_h_var}
\end{figure}

\begin{table}[!ht] 
  \centering
  \begin{tabular}{|c|c||c|c|}
    \hline\rule{0pt}{2.5ex}
    $h$ & \# points per period & $L^{2}$-error & $H^{1}$-error\\
    \hline
    \hline\rule{0pt}{2.5ex}
    $0.1$ &
    $2$ &
    $4.7\times 10^{-1}$ &
    $1.05$ 
    \\
    \hline\rule{0pt}{2.5ex}
    $0.05$ &
    $4$ &
    $5.2\times 10^{-1}$ &
    $1.29$ 
    \\
    \hline\rule{0pt}{2.5ex}
    $0.025$ &
    $8$ &
    $1.82\times 10^{-1}$ &
    $4.3\times 10^{-1}$
    \\
    \hline\rule{0pt}{2.5ex}
    $0.0125$ &
    $16$ &
    $1.89\times 10^{-2}$ &
    $6.4\times 10^{-2}$
    \\
    \hline\rule{0pt}{2.5ex}
    $0.00625$ &
    $32$ &
    $1.41\times 10^{-3}$ &
    $1.00\times 10^{-2}$
    \\
    \hline\rule{0pt}{2.5ex}
    $0.0003125$ &
    $64$ &
    $9.3\times 10^{-5}$ &
    $2.21\times 10^{-3}$
    \\
    \hline\rule{0pt}{2.5ex}
    $0.0015625$ &
    $128$ &
    $6.1\times 10^{-6}$ &
    $5.5\times 10^{-4}$
    \\
    \hline\rule{0pt}{2.5ex}
    $0.00078125$ &
    $256$ &
    $4.6\times 10^{-7}$ &
    $1.36\times 10^{-4}$
    \\
    \hline
  \end{tabular}
  \caption{Relative $L^{2}$ and $H^{1}$ errors between the exact
    solution $\phi^{\varepsilon }$ and the computed solution $\phi _A$
    (AP) for $m=10$ and $\varepsilon = 10^{-10}$ as a function of
    $h$.}
  \label{tab:bp_h_var}
\end{table}

%\subsubsection{3D test case, $b  = const.$}
\subsubsection{3D test case, uniform and aligned $b$-field}
Finally, we test our method on a simple $3D$ test case. Let the 
field $b$ be aligned with the $X$-axis:
\begin{gather}
  b= 
  \left(
    \begin{array}{c}
      1 \\
      0 \\
      0
    \end{array}
  \right)
  \label{eq:Jf0a}.
\end{gather}
Let $\Omega = [0,1] \times [0,1]\times [0,1]$, and the source term $f$
is such that the solution is given by
\begin{gather}
  \phi^{\varepsilon } = \sin \left(\pi y \right)\sin \left(\pi z \right) 
  + \varepsilon \cos \left( 2\pi x\right)
  \sin \left(\pi y \right)\sin \left(\pi z \right), \\
  p^{\varepsilon } = \sin \left(\pi y \right)\sin \left(\pi z \right)\,, \quad
  q^{\varepsilon } = \varepsilon \cos \left( 2\pi x\right)
  \sin \left(\pi y \right)\sin \left(\pi z \right)
  \label{eq:J97a:2}.
\end{gather}
Numerical simulations were performed on a $30\times 30\times 30$
grid. Once again all three methods are compared. The $L^2$ and
$H^{1}$-errors are given on Figure \ref{fig:error3d}. The numerical
results are equivalent with those obtained in the 2D test with
constant $b$. Note that it is difficult to perform 3D simulations with
more refined grids, due to memory requirements on standard desktop
equipment as we are doing now. Every row in the matrix constructed for
the Singular Perturbation model, can contain up to 125 non zero
entries (for $\mathbb Q_2$ finite elements), while matrices associated
with the Asymptotic Preserving reformulation have rows with up to 375
non zero entries. Furthermore the dimension of the latter is five
times bigger. The memory requirements of the direct solver used in our
simulations grow rapidly. The remedy could be to use an iterative
solver with suitable preconditioner. Finding the most efficient method
to inverse these matrices is however beyond the scope of this
paper. In future work we will address this problem as well as a
parallelization of this method.

\def\xxxa{0.45\textwidth}
\begin{figure}[!ht] 
  \centering
  \subfigure[$L^{2}$ error for a grid with $30\times 30 \times 30$ points.]
  {\includegraphics[angle=-90,width=\xxxa]{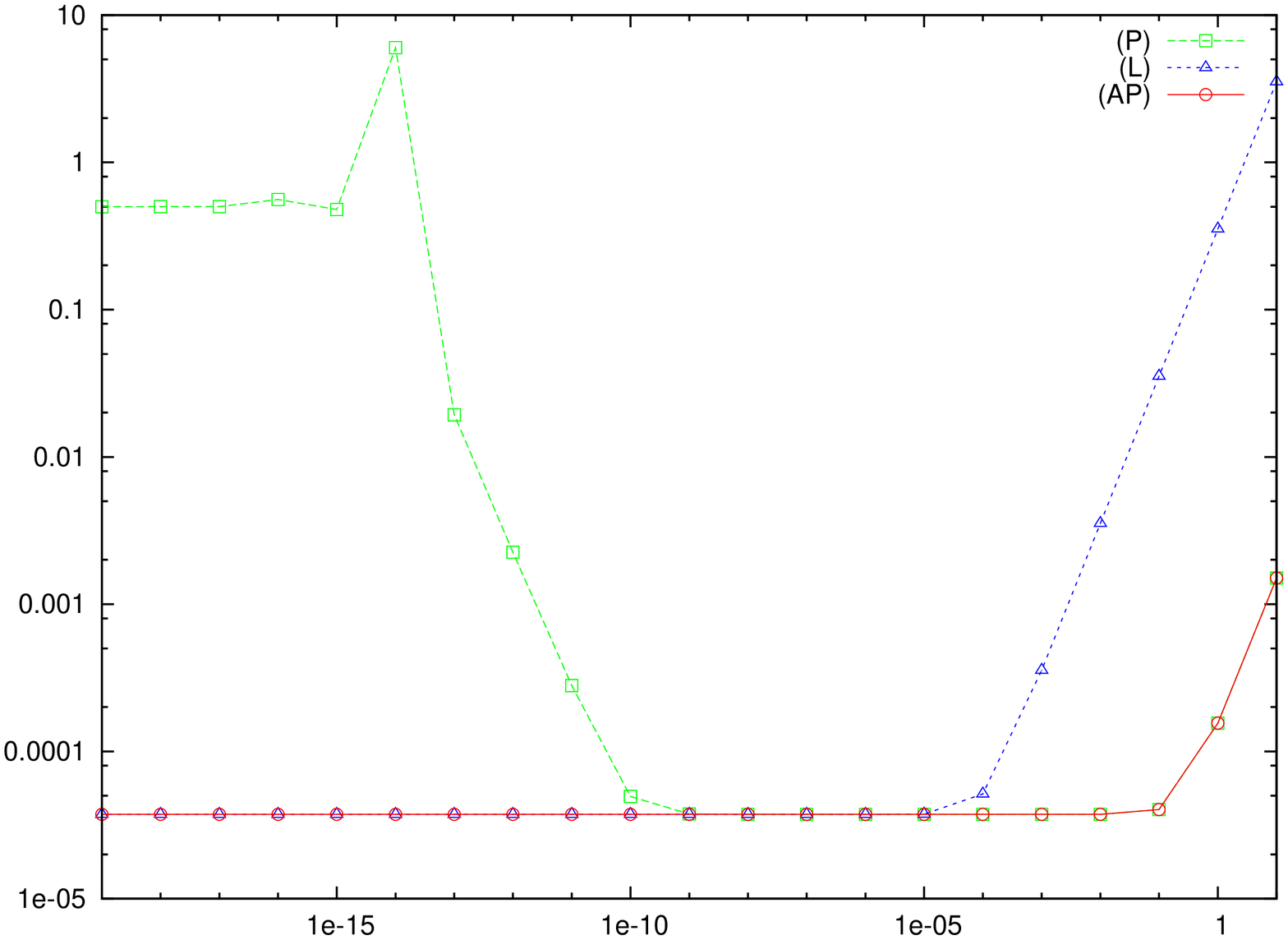}}
  \subfigure[$H^{1}$ error for a grid with $30\times 30 \times 30$ points.]
  {\includegraphics[angle=-90,width=\xxxa]{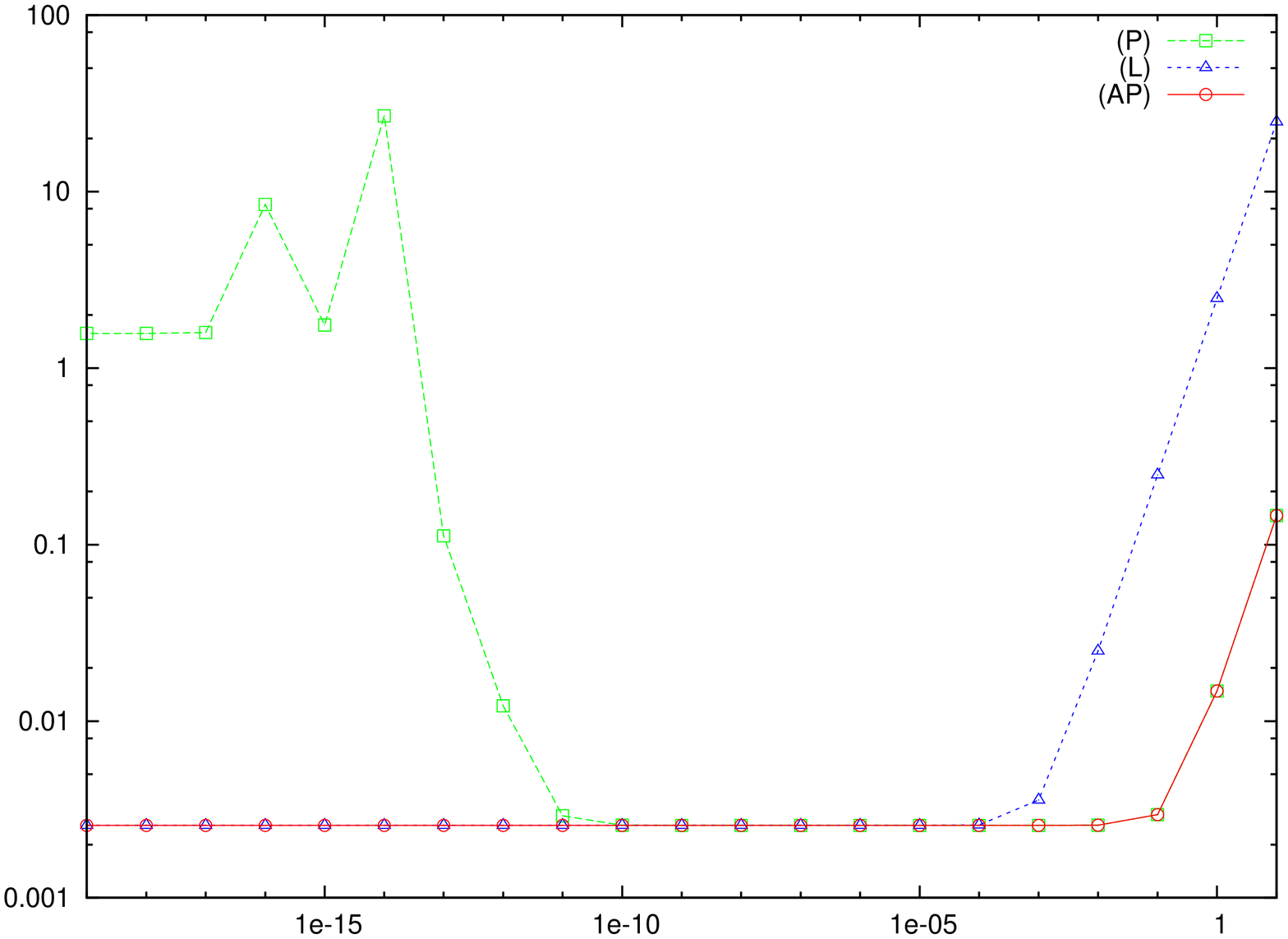}}

  \caption{Absolute $L^{2}$ (left column) and $H^{1}$ (right column)
    errors between the exact solution $\phi^{\varepsilon }$ and the
    computed solution $\phi _A$ (AP), $\phi _L$ (L), $\phi _P$ (P) for
    the 3D test case. The errors are plotted as a function of the
    anisotropy ratio $\varepsilon $.}
  \label{fig:error3d}
\end{figure}

%%%%%%%%%%%%%%%%%%%%%%%%%%%%
\section{Conclusions}\label{SEC6}
%%%%%%%%%%%%%%%%%%%%%%%%%%%%

The asymptotic preserving method presented in this paper is shown to
be very efficient for the solution of highly anisotropic elliptic
equations, where the anisotropy direction is given by an arbitrary,
but smooth vector field $b$ with non-adapted coordinates and
meshes. The results presented here generalize the procedure used in
\cite{DDN} and have the important advantage to permit the use of
Cartesian grids, independently on the shape of the
anisotropy. Moreover, the scheme is equally accurate, independently on
the anisotropy strength, avoiding thus the use of coupling
methods. The numerical study of this AP-scheme shall be investigated
in a forthcoming paper, in particular the $\eps$-independent
convergence results shall be stated.

Another important related work consists in extending our methods to
the case of anisotropy ratios $\eps$, which are variable in $\Omega$ from
moderate to very small values. This is important, for example, in
plasma physics simulations as already noted in the introduction. An
alternative strategy to the Asymptotic Preserving schemes whould be to
couple a standard discretization in subregions with moderate $\eps$
with a limit ($\eps\to 0$) model in subregions with small $\eps$ as
suggested, for example, in \cite{BCDDGT_1,Keskinen}. However, the
limit model is only valid for $\varepsilon \ll 1$ and cannot be
applied for weak anisotropies. Thus, the coupling strategy requires
existence of a range of anisotropy strength where both methods are
valid. This is rather undesirable since this range may not exist at
all, as illustrated by our results in Fig. \ref{fig:error}.

\appendix
%%%%%%%%%%%%%%%%%%%%%%%%%%%%%%%%%%%%%%%%%%%%%%%
\section{Decompositions $\mathcal{V}=\mathcal{G}\oplus ^{\perp }\mathcal{A}$, $\tilde{\mathcal{V}}=\tilde{\mathcal{G}}\oplus \mathcal{L}$ and related estimates} \label{appA}
%%%%%%%%%%%%%%%%%%%%%%%%%%%%%%%%%%%%%%%%%%%%%%%

\color{black}
We shall show in this Appendix that all the statements in Hypotheses B and B' can be rigorously derived under some 
assumptions on the domain boundary $\partial\Omega$ \fabrice{and on} the manner in which it is intersected by the field $b$. 
We assume essentially that $b$ is tangential to $\partial\Omega$ on $\partial\Omega_D$ and that $b$ penetrates the remaining part of the boundary $\partial\Omega_N$ at an angle 
that stays away from 0 on $\overline{\partial\Omega}_N$. We assume also that $\partial\Omega_N$ consists of two disjoint components for which there exist global and smooth parametrizations. 
This last assumption can be weakened (existence of an atlas of local smooth parametrizations should be sufficient) at the expense of lengthening the proofs. 
The precise set of our assumptions is the following:\\
\color{black}

\textbf{Hypothesis C} The boundary of $\Omega $ is the union of three
components: $\partial \Omega _{D}$\textit{\ }where $b\cdot n=0$, $\partial
\Omega _{in}$\textit{\ }where $b\cdot n\leq -\alpha $ and $\partial \Omega
_{out}$\textit{\ }where $b\cdot n\geq \alpha $ with some constant $\alpha >0$. 
Moreover, there is a smooth system of coordinates $\xi _{1},\ldots ,\xi
_{d-1}$ on $\partial \Omega _{in}$ meaning that there is a bounded domain $%
\Gamma _{in}\in \mathbb{R}^{d-1}$ and a one-to-one map $h_{in}:\Gamma
_{in}\rightarrow \mathbb{R}^{d}$ such that $h_{in}\in C^{2}(\overline{\Gamma}%
_{in})$ and  $\partial \Omega _{in}$ is the image of $h_{in}(\xi _{1},\ldots
,\xi _{d-1})$ as $(\xi _{1},\ldots ,\xi _{d-1})$ goes over $\Gamma _{in}$.
The matrix formed by the vectors $(\partial h_{in}/\partial \xi _{1},\ldots
,\partial h_{in}/\partial \xi _{1},n)$ is invertible for all $(\xi
_{1},\ldots ,\xi _{d-1})\in \overline{\Gamma}_{in}$. Similar assumptions hold
also for $\partial \Omega _{out}$ (changing $\Gamma _{in}$ to $\Gamma _{out}$
and  $h_{in}$ to $h_{out}$).\\

\bigskip Using this hypothesis we can introduce a system of coordinates in $\Omega$ 
such that the field lines of $b$ coincide with the coordinate lines. 
To do this consider the initial value problem for a parametrized ordinary
differential equation (ODE):
\begin{equation}
\frac{\partial X}{\partial \xi _{d}}(\xi ^{\prime },\xi _{d})=b(X(\xi
^{\prime },\xi _{d})),~X(\xi ^{\prime },0)=h_{in}(\xi ^{\prime }).
\label{odeX}
\end{equation}%
Here $X(\xi ^{\prime },\xi _{d})$ is $\mathbb{R}^{d}$-valued and $\xi
^{\prime }$ stands for $(\xi _{1},\ldots ,\xi _{d-1})$. For any
fixed $\xi ^{\prime }\in \Gamma _{in}$, $\ $equation (\ref{odeX}) should be
understood as an ODE for a function of $\xi _{d}$. Its solution $X(\xi
^{\prime },\xi _{d})$ goes then over the field line of $b$ starting (as $\xi
_{d}=0$) at the point on the inflow boundary $\partial \Omega _{in}$, parametrized by $%
\xi ^{\prime }$. This field line hits the outflow boundary $\partial \Omega
_{out}$ somewhere. In other words, for any $\xi ^{\prime }\in \Gamma _{in}$
there exists $L(\xi ^{\prime })>0$ such that $X(\xi ^{\prime },L(\xi
^{\prime }))\in \partial \Omega _{out}$. The domain of definition of $X$ is
thus%
\begin{equation*}
D=\{(\xi ^{\prime },\xi _{d})\in \mathbb{R}^{d}~/~\xi ^{\prime }\in \Gamma
_{in}\text{ and }0<\xi _{d}<L(\xi ^{\prime })\}.
\end{equation*}%
Gathering the results on parametrized ODEs, from for instance \cite{ODE}, we
conclude that $X(\xi ^{\prime },\xi _{d})=X(\xi _{1},\ldots ,\xi _{d})$ is a
smooth function of all its $d$ parameters, more precisely $X\in C^{2}(\overline{D}%
)$. Evidently, the map $X$ is one-to-one from $\overline{D}$ to $\overline{\Omega}$
and thus $\xi _{1},\ldots ,\xi _{d}$ provide a system of coordinates for $%
\overline{\Omega}$. Moreover this system is not degenerate in the sense that the
vectors $\partial X/\partial \xi _{1},\ldots ,\partial X/\partial \xi _{d}$
are linearly independent at each point of $\overline{\Omega}$. Indeed, if this
was not the case, then there would exist a non trivial linear combination  $%
\lambda _{1}\partial X/\partial \xi +\cdots +\lambda _{d}\partial X/\partial
\xi _{d}$ that would vanish at some point in $\overline{\Omega}$. But, ODE (\ref%
{odeX}) implies 
\begin{equation*}
\frac{\partial }{\partial \xi _{d}}\sum_{i=1}^{d}\lambda _{i}\frac{\partial X%
}{\partial \xi _{i}}(\xi ^{\prime },\xi _{d})=\nabla b(X(\xi ^{\prime },\xi
_{d}))\cdot \sum_{i=1}^{d}\lambda _{i}\frac{\partial X}{\partial \xi _{i}}%
(\xi ^{\prime },\xi _{d})
\end{equation*}%
so that, the unique solution of this ODE, i.e. the linear combination $\sum_{i=1}^{d}\lambda _{i}\frac{\partial X}{%
\partial \xi _{i}}$, would vanish on the whole field line, in particular on
the inflow. But this is impossible since $\frac{\partial X}{\partial \xi _{i}%
}=\frac{\partial h_{in}}{\partial \xi _{i}}$, $i=1,\ldots ,d-1$ on the
inflow, while $\frac{\partial X}{\partial \xi _{d}}=b$ and the vectors $%
\left( \frac{\partial h_{in}}{\partial \xi _{1}},\ldots ,\frac{\partial
h_{in}}{\partial \xi _{d-1}},b\right) $ are linearly independent for all $%
(\xi _{1},\ldots ,\xi _{d-1})\in \overline{\Gamma}_{in}$. We see thus that the
Jacobian $J=\det (\partial X_{j}/\partial \xi _{i})$ does not vanish on $%
\overline{\Omega}$ so that we can assume that $m<J<M$ everywhere on $\overline{\Omega}$
with some positive constants $m$ and $M$ (assuming that  $J$ is positive
does not prevent the generality since if \ $J$ is negative in $\Omega $ than
one can replace $\xi _{1}$ by $-\xi _{1}$).  Since $X \in C^2(\overline{\Omega})$, we have also that $J \in C^1(\overline{\Omega})$.

One also sees easily that the top of $D$ is given by a smooth function $%
L(\xi ^{\prime })$. Indeed, $L(\xi ^{\prime })$ is determined for each $\xi
^{\prime }\in \Gamma _{in}$ from the equation $X(\xi ^{\prime },L(\xi
^{\prime }))=h_{out}(\eta )$ with some $\eta =(\eta _{1},\ldots ,\eta
_{d-1})\in \Gamma _{out}$. We know already that this equation is solvable
for $\xi _{d}=L(\xi ^{\prime })$, $\eta _{1},\ldots ,\eta _{d-1}$ for any $%
\xi ^{\prime }\in \Gamma _{in}$. To conclude that the solution depends
smoothly on $\xi ^{\prime }$ we can apply the implicit function theorem to
the equation 
\begin{equation*}
F(\xi ^{\prime };\xi _{d},\eta _{1},\ldots ,\eta _{d-1})=X(\xi ^{\prime
},\xi _{d})-h_{out}(\eta _{1},\ldots ,\eta _{d-1})=0.
\end{equation*}%
Indeed, the $R^{d}$-valued function $F$ is smooth and the matrix of its
partial derivatives with respect to $\xi _{d},\eta _{1},\ldots ,\eta _{d-1}$
is invertible since $\partial F/\partial \xi _{d}=b$ and $\partial
F/\partial \eta _{i}=-\partial h_{out}/\partial \eta _{i}$ at some point at
the outflow and the vectors $\partial h_{out}/\partial \eta _{i}$ lie in the
tangent plane to $\partial \Omega _{out}$ while $b$ is nowhere in this
plane. We have moreover that $L\in C^{1}(\overline{\Gamma}_{in})$. Indeed, we can
prove that all the derivatives of $L$ are bounded. In order to do it, let us
remark that the differential of $X(\xi ^{\prime },L(\xi ^{\prime }))$
represents a vector in the tangent plane at some point on $\partial \Omega
_{out}$ so that it is perpendicular to the outward normal $n$. We have thus
for any $i=1,\ldots ,d-1$ 
\begin{equation*}
0=n\cdot \left( \frac{\partial X}{\partial \xi _{i}}(\xi ^{\prime },L(\xi
^{\prime }))+\frac{\partial X}{\partial \xi _{d}}(\xi ^{\prime },L(\xi
^{\prime }))\frac{\partial L}{\partial \xi _{i}}(\xi ^{\prime })\right)
=n\cdot \left( \frac{\partial X}{\partial \xi _{i}}(\xi ^{\prime },L(\xi
^{\prime }))+b(\xi ^{\prime },L(\xi ^{\prime }))\frac{\partial L}{\partial
\xi _{i}}(\xi ^{\prime })\right) 
\end{equation*}%
so that 
\begin{equation*}
\frac{\partial L}{\partial \xi _{i}}(\xi ^{\prime })=-\frac{n\cdot \frac{%
\partial X}{\partial \xi _{i}}(\xi ^{\prime },L(\xi ^{\prime }))}{n\cdot
b(\xi ^{\prime },L(\xi ^{\prime }))}
\end{equation*}%
and this is bounded since $X$ has bounded partial derivatives and $n\cdot
b\geq \alpha $ by the hypothesis. Note also that $L$ is strictly positive.

\begin{itemize}
\item
We can now establish the decomposition $\mathcal{V}=\mathcal{G}\oplus
^{\perp }\mathcal{A}$. Take any $\phi \in \mathcal{V}$ $\cap C^{1}(\overline{%
\Omega})$ and introduce $p\in L^{2}(\Omega )$ by%
\begin{equation}
p(x)=p(\xi ^{\prime },\xi _{d})=p(\xi ^{\prime })=\frac{\int_{0}^{L(\xi
^{\prime })}\phi (\xi ^{\prime },t)J(\xi ^{\prime },t)dt}{\int_{0}^{L(\xi
^{\prime })}J(\xi ^{\prime },t)dt}.  \label{pdef}
\end{equation}%
(from now on we switch back and forth between the Cartesian coordinates $%
x=(x_{1},\ldots ,x_{d})$ and the new ones $(\xi ^{\prime },\xi _{d})=(\xi
_{1},\ldots ,\xi _{d}))$. Evidently, $p$ is constant along each field line$.$ 
 Moreover, $p$ is the $L^{2}$-orthogonal projection of $\phi $ on the space
of such functions$.$ Indeed, if $\psi =\psi (\xi ^{\prime })\in L^{2}(\Omega )
$ is any function constant along each field line then%
\begin{eqnarray*}
\int_{\Omega }p\psi dx &=&\int_{D}p\psi Jd\xi =\int_{\Gamma _{in}}p(\xi
^{\prime })\psi (\xi ^{\prime })\int_{0}^{L(\xi ^{\prime })}J(\xi ^{\prime
},\xi _{d})d\xi _{d}d\xi ^{\prime } \\
&=&\int_{\Gamma _{in}}\int_{0}^{L(\xi ^{\prime })}\phi (\xi ^{\prime },\xi_d)\, \psi
(\xi ^{\prime })\,J(\xi ^{\prime },\xi _{d})d\xi _{d}d\xi ^{\prime
}=\int_{\Omega }\phi \psi dx.
\end{eqnarray*}%
Let us prove that $p\in \mathcal{V}$, i.e. that its derivatives are square
integrable. The change of variable $t=L(\xi')s$ yields the function
$$
p(\xi ^{\prime })=\frac{\int_{0}^1 \phi (\xi ^{\prime },L(\xi')s)J(\xi ^{\prime },L(\xi')s)ds}{\int_{0}^1 J(\xi ^{\prime },L(\xi')s)ds}. 
$$
Now we have  $\partial p/\partial \xi _{d}=0$ and  for all  $%
\partial p/\partial \xi _{i,i=1,\ldots ,d-1}$ denoting $a=a(\xi ^{\prime
})=(\int_{0}^{1}J(\xi ^{\prime },L(\xi')s)ds)^{-1}$, $\phi = \phi(\xi ^{\prime },L(\xi')s)$ and same for $J$ we obtain
\begin{equation}
\begin{array}{lll}
\ds \frac{\partial p}{\partial \xi _{i}}&=&\ds \frac{\partial a}{\partial \xi _{i}}%
\int_{0}^{1}\phi Jds+a\int_{0}^{1}\frac{\partial \phi }{\partial \xi _{i}}Jds+a\int_{0}^{1}\frac{\partial \phi }{\partial \xi _{d}}%
\, \frac{\partial L }{\partial \xi_i} s \,J\, ds\\[3mm]
&&\ds + a\int_{0}^{1} \phi \, \frac{\partial J }{\partial \xi _{i}} ds+a\int_{0}^{1} \phi \, \frac{\partial J }{\partial \xi _{d}} \, 
\frac{\partial L }{\partial \xi_i} s \,ds
\end{array}
\end{equation}%
Using all the previous bounds on the functions $L$ and $J$ and skipping the
details of somewhat tedious calculations, we arrive at%
\begin{eqnarray*}
\int_{\Omega }\left( \frac{\partial p}{\partial \xi _{i}}\right) ^{2}dx
&=&\int_{\Gamma _{in}}\int_{0}^{L(\xi ^{\prime })}\left( \frac{\partial p}{%
\partial \xi _{i}}\right) ^{2}Jd\xi _{d}d\xi ^{\prime } \\
&\leq &C\int_{\Gamma _{in}}\int_{0}^{L(\xi ^{\prime })}\left( \phi
^{2}+\left( \frac{\partial \phi }{\partial \xi _{i}}\right) ^{2}+\left( \frac{\partial \phi }{\partial \xi _{d}}\right) ^{2}\right)
Jd\xi _{d}d\xi ^{\prime }
\end{eqnarray*}%
implying
\begin{equation*}
\left\Vert \frac{\partial p}{\partial \xi _{i}}\right\Vert _{L^{2}(\Omega
)}^{2}\leq C\left( \left\Vert \phi \right\Vert _{L^{2}(\Omega
)}^{2}+\left\Vert \frac{\partial \phi }{\partial \xi _{i}}\right\Vert
_{L^{2}(\Omega )}^{2}+\left\Vert \frac{\partial \phi }{\partial \xi _{d}}\right\Vert
_{L^{2}(\Omega )}^{2}\right) \le C\left\Vert \phi \right\Vert
_{H^{1}(\Omega )}^2\,.
\end{equation*}%
Thus $p\in H^{1}(\Omega )$, hence $p\in \mathcal{G}$ and 
$q=\phi -p\in \mathcal{A}$. Since the dependence of $p$ on $\phi $ is
continuous in the norm of $H^{1}(\Omega )$, a density argument shows that the
decomposition $\phi =p+q$ with $p\in \mathcal{G}$ and $q\in \mathcal{A}$
exists for any $\phi \in \mathcal{V}$. 

\item Let us now introduce the operator $P$ as the $L^{2}$-orthogonal  projector on $%
\mathcal{G}$, that means
$$
P: {\cal V} \rightarrow {\cal G}\,, \quad \phi \in {\cal V} \longmapsto P \phi \in {\cal G}\quad \textrm{given by} \quad (\ref{pdef})\,. 
$$
Then, the estimates in
the preceding paragraph show that the operator $P$ is continuous in the norm of $%
H^{1}(\Omega )$:%
\begin{equation}
||\nabla _{\perp }(P\phi )||_{L^{2}(\Omega )}\leq C||\nabla \phi
||_{L^{2}(\Omega )}\,,\quad \forall \phi \in \mathcal{V}  \label{reg}
\end{equation}

\item We have also the following Poincar\'{e}-Wirtinger inequality: 
\begin{equation}
||\phi -P\phi ||_{L^{2}(\Omega )}\leq C||\nabla _{||}\phi ||_{L^{2}(\Omega
)}\,,\quad \forall \phi \in \mathcal{V}\,.  \label{PoinW_bis}
\end{equation}%
To prove this, it is sufficient to establish that $||q||_{L^{2}(\Omega
)}\leq C||\nabla _{||}q||_{L^{2}(\Omega )}$ for all $q\in \mathcal{A}$. We
observe that   
\begin{equation*}
||q||_{L^{2}(\Omega )}^{2}=\int_{\Gamma _{in}}\int_{0}^{L(\xi ^{\prime
})}q^{2}(\xi ^{\prime },\xi _{d})J(\xi ^{\prime },\xi _{d})d\xi _{d}d\xi
^{\prime }
\end{equation*}%
and 
\begin{equation*}
||\nabla _{||}\phi ||_{L^{2}(\Omega )}^{2}=\int_{\Gamma
_{in}}\int_{0}^{L(\xi ^{\prime })}\left( \frac{\partial q}{\partial \xi _{d}}%
\right) ^{2}(\xi ^{\prime },\xi _{d})J(\xi ^{\prime },\xi _{d})d\xi _{d}d\xi
^{\prime }\,.
\end{equation*}%
The requirement $q\in \mathcal{A}$ is equivalent to 
\begin{equation}
\int_{0}^{L(\xi ^{\prime })}q(\xi ^{\prime },\xi _{d})J(\xi ^{\prime },\xi
_{d})d\xi _{d}=0  \label{qrest}\quad \textrm{f.a.a.} \,\, \xi ^{\prime }\in \Gamma _{in}\,.
\end{equation}
We have thus to prove for
every $\xi ^{\prime }$ 
\begin{equation*}
\int_{0}^{L(\xi ^{\prime })}q^{2}(\xi ^{\prime },\xi _{d})J(\xi ^{\prime
},\xi _{d})d\xi _{d}\leq C^{2}\int_{0}^{L(\xi ^{\prime })}\left( \frac{%
\partial q}{\partial \xi _{d}}\right) ^{2}(\xi ^{\prime },\xi _{d})J(\xi
^{\prime },\xi _{d})d\xi _{d}
\end{equation*}%
provided (\ref{qrest}). Fixing any $\xi ^{\prime },$ making the change of
integration variable $\xi _{d}=L(\xi ^{\prime })t$ and introducing the functions $%
u(t)=q(\xi ^{\prime },L(\xi ^{\prime })t)J(\xi ^{\prime },L(\xi ^{\prime })t)
$ and $J(t)=J(\xi ^{\prime },L(\xi ^{\prime })t)$, we rewrite the last inequality
as 
\begin{equation}
\int_{0}^{1}\frac{u^{2}(t)}{J(t)}dt\leq \frac{C^{2}}{L^{2}(\xi ^{\prime })}%
\int_{0}^{1}\left( \frac{u^{\prime }(t)}{J(t)}-\frac{u(t)}{J^{2}(t)}%
J^{\prime }(t)\right) ^{2}J(t)dt\,.  \label{uine}
\end{equation}%
Since $\int_{0}^{1}u(t)dt=0$ we have by the standard Poincar\'{e} inequality% 
\begin{equation}
\int_{0}^{1}u^{2}(t)dt\leq C_{P}^{2}\int_{0}^{1}\left( u^{\prime }(t)\right)
^{2}dt\,.  \label{uinep}
\end{equation}%

\item 
Let us turn to the verification of Hypothesis B'. Take any $u\in \mathcal{\tilde{%
V}}$. We want to prove that one can decompose $u=p+q$ with  $p\in \mathcal{%
\tilde{G}}$ and $q\in \mathcal{L}$ and the trace of $u$ on $\partial \Omega
_{in}$ (denoted $g$) is in $L^{2}(\partial \Omega _{in})$. In the $\xi $%
-coordinates we can write a surface  element of $\partial \Omega _{in}$ as $%
d\sigma =S(\xi ^{\prime })d\xi ^{\prime }$ with a function $S$ smoothly
depending on $\xi ^{\prime }$. We see now that for $u$ suffuciently smooth%
\begin{eqnarray*}
||g||_{L^{2}(\partial \Omega _{in})}^{2} 
&=&\int_{\Gamma _{in}}g^{2}(\xi^{\prime })S(\xi ^{\prime })d\xi ^{\prime }\\
&\leq& C\int_{0}^{1}\int_{\Gamma
_{in}}\left[ u^{2}(\xi ^{\prime },L(\xi ^{\prime })s)+\frac{1}{L(\xi
^{\prime })}\left( \frac{\partial u}{\partial \xi _{d}}\right) ^{2}(\xi
^{\prime },L(\xi ^{\prime })s)\right] S(\xi ^{\prime })d\xi ^{\prime }ds \\
&&\quad \text{(by a one-dimensional trace inequlity)} \\
&\leq &C||u||_{\mathcal{\tilde{V}}}^{2}
\end{eqnarray*}%
By density, the trace $g$ is thus defined for any $u\in \mathcal{\tilde{V}}$
with $||g||_{L^{2}(\partial \Omega _{in})}\leq C||u||_{\mathcal{\tilde{V}}}$%
. Taking $p=p(\xi ^{\prime })=g(\xi ^{\prime })$ we observe by a similar
calculation that $||p||_{L^{2}( \Omega)}\leq C||u||_{\mathcal{\tilde{V}}%
}$ so that $p\in \mathcal{\tilde{G}}$. By definition $q=u-p\in \mathcal{L}$.

\end{itemize}

\section{On the choice of the finite element space
  $\mathcal{L}_h$}\label{appB}

Let $\Omega$ be the rectangle $(0,L_x)\times(0,L_y)$ and the
anisotropy direction be constant and aligned with the $y$-axis:
$b=(0,1)$. Let us use the $\mathbb Q_k$ finite elements on a Cartesian
grid, i.e. take some basis function $\theta_{x_i}(x)$,
$i=0,\ldots,N_x$ and $\theta_{y_j}(y)$, $j=0,\ldots,N_y$ and define
the complete finite element space $X_h$ (without any restrictions on
the boundary) as span$\{\theta_{x_i}(x)\theta_{y_j}(y)\ 0\le i\le
N_x,\ 0\le j\le N_y\}$. The following subspace is then used for the
approximation of the unknowns $p,q,l \in {\cal V}$
$$
\mathcal{V}_h=\{v_h\in X_h/v_h|_{\partial\Omega_{D}}=0\}.
$$
We want to prove that taking for the approximation of $\lambda,\mu \in {\cal L}$ the space $\mathcal{L}_h$ under the form 
\be \label{L_H_bad}
\mathcal{L}_h=\{\lambda_h\in X_h/\lambda_h|_{\partial\Omega_{in}}=0\}\,,
\ee
leads to an ill posed problem (\ref{eq:Jt8a}).\\

\noindent\textbf{Claim} There exist $\lambda_h\in\mathcal{L}_h$, $\lambda_h \neq 0$, such that $a_{||}(\lambda_h,p_h)=0$ for all $p_h\in\mathcal{V}_h$. In fact 
there are exactly $2N_y$ linearly independent functions having this property.

\bigskip
\begin{rem} In the continuous case, the equation
$$
a_{||}(p,\lambda)=0\,, \quad \forall p \in {\cal V}\,,
$$
implies $\lambda = 0$ by density arguments. These density arguments are lost when discretizing the spaces ${\cal V}$ resp. ${\cal L}$. 
\end{rem}

\noindent\textbf{Proof of the Claim.} We can suppose that the basis functions $\theta_{ij}(x,y):=\theta_{x_i}(x)\theta_{y_j}(y)$ are 
enumerated so that $\theta_{ij}(0,y)=0$ for all $i\ge 1$ and $\theta_{0j}(0,y)\not=0$. 
Hence for all $p_h=\sum p_{ij}\theta_{ij}\in\mathcal{V}_h$, the coefficients satisfy $p_{0j}=0$
since the part of the boundary $\{x=0\}$ is in $\partial\Omega_{D}$. 
Let $M=(m_{ik})_{0\le i,k \le N_x}$ be the mass matrix in the $x$-direction: $m_{ik}=\int\theta_{x_i}(x)\theta_{x_k}(x)dx$. 
This matrix is invertible, hence there is a vector $a\in\mathbb{R}^{N_x+1}$ that solves $Ma=e$ with $e\in\mathbb{R}^{N_x+1}$, $e=(1,0,\ldots,0)^t$.
Take any fixed integer $j$, $1\le j\le N_y$ and define $\lambda_h\in\mathcal{L}_h$ as $\lambda_h=\sum a_i\theta_{ij}$. Then,
for all $p_h=\sum p_{kl}\theta_{kl}\in\mathcal{V}_h$ we have
\begin{align*} 
a_{||}(\lambda_h,p_h)
&=\sum_{i,k,l} a_ip_{kl}\int_\Omega\frac{\partial\theta_{ij}}{\partial y}\frac{\partial\theta_{kl}}{\partial y}dxdy\\
&=\sum_{i,k,l} a_ip_{kl} \int_0^{L_x}\theta_{x_i}(x)\theta_{x_k}(x)dx \int_0^{L_y}\theta'_{y_j}(y)\theta'_{y_l}(y)dy\\
&=\sum_{k,l} \delta_{k0}p_{kl}  \int_0^{L_y}\theta'_{y_j}(y)\theta'_{y_l}(y)dy = 0.
\end{align*} 
As we can do this for all  $(i,j)$, $i=0$, $1\le j\le N_y$ and in the same manner for all $(i,j)$, $i=N_x$, $1\le j\le N_y$, there are
$2N_y$ linearly independent functions with the property $a_{||}(\lambda_h,p_h)=0$ for all $p_h\in\mathcal{V}_h$. 

We see now that the system (\ref{eq:Jt8a}) with zero right hand side $f=0$ possesses non-zero solutions
$(p^\varepsilon_h,\;\lambda^\varepsilon_h,\;q^\varepsilon_h,\;l^\varepsilon_h,\;\mu^\varepsilon_h)
=(0,\lambda^j_h,0,0,0)$ where $\lambda^j_h$ is any of the functions constructed in the preceding paragraph. It means that (\ref{eq:Jt8a})
is ill posed, i.e. the corresponding matrix is singular.

\medskip

{\bf Acknowledgement.} This work has been supported by the Marie Curie Actions of the European 
Commission under the contract DEASE (MEST-CT-2005-021122),  by the French 
'Commissariat  \`a l'Energie Atomique (CEA)' under contracts ELMAG 
(CEA-Cesta 4600156289), and GYRO-AP (Euratom-CEA V 3629.001), by the Agence Nationale 
de la Recherche (ANR) under contract IODISEE (ANR-09-COSI-007-02), by the
'Fondation Sciences et Technologies pour l'A\'eronautique et l'Espace (STAE)'
under contract PLASMAX (RTRA-STAE/2007/PF/002) and by the Scientific
Council of the Universit\'e Paul Sabatier, under contract MOSITER. Support from the 
French magnetic fusion programme 'f\'ed\'eration de recherche sur la fusion 
par confinement magn\'etique' is also acknowledged. The authors would like
to express their gratitude to G. Gallice and C. Tessieras from CEA-Cesta for bringing 
their attention to this problem, to G. Falchetto, X. Garbet and M. Ottaviani from
CEA-Cadarache, for their constant support to this research programme.  

\medskip

\bibliographystyle{abbrv}
\bibliography{bib_aniso}
 
\end{document}